\documentclass[11pt]{amsart}
\usepackage{amsmath,amsfonts,amsthm,amsopn,amssymb,latexsym,soul}
\usepackage{cite}
\usepackage{color,enumitem,graphicx,esint}
\usepackage{upgreek,bm} 
\usepackage[colorlinks=true,urlcolor=blue,
citecolor=red,linkcolor=blue,linktocpage,pdfpagelabels,
bookmarksnumbered,bookmarksopen]{hyperref}
\usepackage[english]{babel}
\usepackage[left=2.61cm,right=2.61cm,top=2.72cm,bottom=2.72cm]{geometry}
\usepackage[hyperpageref]{backref}
\usepackage[colorinlistoftodos]{todonotes}

\makeatletter
\providecommand\@dotsep{5}
\def\listtodoname{List of Todos}
\def\listoftodos{\@starttoc{tdo}\listtodoname}
\makeatother

\numberwithin{equation}{section}
\newtheorem{theorem}{Theorem}[section]
\newtheorem{lemma}[theorem]{Lemma}

\newtheorem{definition}[theorem]{Definition}
\newtheorem{proposition}[theorem]{Proposition}
\newtheorem{remark}[theorem]{Remark}
\newtheorem{corollary}[theorem]{Corollary}

\newcommand{\R}{\mathbb{R}}

\newcommand{\RN}{{\mathbb{R}^N}}
\newcommand{\LN}{{\mathbb{L}^{N+1}}}

\DeclareMathOperator{\dist}{dist} 
\DeclareMathOperator{\supp}{supp} 
\renewcommand{\le}{\leqslant}

\renewcommand{\d }{\delta }

\newcommand{\n }{\nabla }
\newcommand{\e }{\epsilon }

\renewcommand{\L}{\mathbb{L}}

\newcommand{\X}{\mathcal{X}}

\newcommand{\N}{\mathbb{N}}

\newcommand{\D }{D^{1,2}(\RN)}
\newcommand{\irn }{\int_{\RN}}

\def\bbm[#1]{\mbo\X{\boldmath $#1$}}
\newcommand{\beq }{\begin{equation}}
\newcommand{\eeq }{\end{equation}}

\def\Xint#1{\mathchoice 
  {\XXint\displaystyle\textstyle{#1}}% 
  {\XXint\textstyle\scriptstyle{#1}}% 
  {\XXint\scriptstyle\scriptscriptstyle{#1}}% 
  {\XXint\scriptscriptstyle\scriptscriptstyle{#1}}% 
  \!\int} 
\def\XXint#1#2#3{{\setbox0=\hbox{$#1{#2#3}{\int}$} 
  \vcenter{\hbox{$#2#3$}}\kern-.5\wd0}} 
\def\Mint{\Xint -}

\begin{document}

\title[A sharp gradient estimate  for the prescribed mean curvature equation]
{A sharp gradient estimate and $W^{2,q}$ regularity for the prescribed mean curvature equation in the Lorentz-Minkowski space}
\author{D\MakeLowercase{enis} Bonheure, A\MakeLowercase{lessandro} Iacopetti}

\thanks{Research partially supported by FNRS and by Gruppo Nazionale per l'Analisi Matematica, la Pro\-ba\-bi\-li\-t\`a e le loro Applicazioni (GNAMPA) of the Istituto Nazionale di Alta Matematica (INdAM). Denis Bonheure is supported by the Fondation Francqui as a Francqui Research Professor and by an ARC Avanc\'e 2020 at ULB.
}
\address[Denis Bonheure]{Research Francqui Professor, D\'epartement de math\'ematique, Universit\'e Libre de Bruxelles, Campus de la Plaine - CP214 boulevard du Triomphe, 1050 Bruxelles, Belgium}
\email{denis.bonheure@ulb.be}
\address[Alessandro Iacopetti]{Dipartimento di Matematica ``Giuseppe Peano'', Universit\`a degli Studi di Torino, via Carlo Alberto, 10 -- 10123 Torino, Italy}
\email{alessandro.iacopetti@unito.it}

\begin{abstract}
We consider the prescribed mean curvature equation for entire spacelike hypersurfaces in the Lorentz-Minkowski space, namely
\begin{equation*}
-\operatorname{div}\left(\displaystyle\frac{\nabla u}{\sqrt{1-|\nabla u|^2}}\right)= \rho \quad \hbox{in }\mathbb{R}^N,
\end{equation*} 
where $N\geq 3$. We first prove a new gradient estimate for classical solutions with smooth data $\rho$. As a consequence we obtain that the unique weak solution of the equation satisfying a homogeneous boundary condition at infinity is locally of class $W^{2,q}$ and strictly spacelike in $\RN$, provided that $\rho\in L^q(\mathbb{R}^N) \cap L^m(\RN)$ with $q>N$ and $m\in[1,\frac{2N}{N+2}]$. 
\end{abstract}

\maketitle
\tableofcontents

\section{Introduction}

\noindent

The Lorentz-Minkowski space, denoted by $\LN$, is defined as the vector space $\R^{N+1}$ endowed with the symmetric bilinear form
\beq\label{def:Lorentflatzmetric}
(\mathbf{x},\mathbf{y})_{\LN} := x_1y_1+\ldots+x_Ny_N-x_{N+1}y_{N+1}.
\eeq
A hypersurface $M\subset \LN$ is said to be spacelike if the induced metric on the tangent space $T_{\mathbf{p}}M$, i.e. the restriction of $(\cdot,\cdot)_{\LN}$ to $T_{\mathbf{p}}M$, is positive definite at any point ${\mathbf{p}}\in M$.

We consider spacelike hypersurfaces that can be expressed globally as cartesian graphs, also called \emph{vertical graphs}. In particular, we focus on the case of entire vertical graphs, namely cartesian graphs of functions defined in the whole of $\RN$. We recall that if $u\in C^1(\RN)$ and $M=\mathrm{graph}(u)$ is the associated cartesian graph then $M$ is spacelike if and only if $|\nabla u|<1$ in $\RN$ (see Section 2), and in that case $u$ is said to be \emph{strictly spacelike}.  The prescribed mean curvature equation for entire spacelike vertical graphs in the Lorentz-Minkowski space is 
\begin{equation}\label{eq:BI2}
-\operatorname{div}\left(\displaystyle\frac{\nabla u}{\sqrt{1-|\nabla u|^2}}\right)= \rho \quad \hbox{in }\mathbb{R}^N,
\end{equation} 
where $\rho:\RN \to \R$ is a given function that prescribes the mean curvature pointwise. 
 Maximal, that is when $\rho = 0$, and constant mean curvature spacelike hypersurfaces are important in general relativity as first emphasized in the fundamental paper of Lichnerowicz \cite{Lich}. They were later used to study dynamical aspects of general relativity or the structure of the singularities in the space of solutions of Einstein’s equations and in the first proof of the positivity of gravitational mass. We refer to \cite{Ba2,Ba3,ChY,FMM,MAR,SY1} and the many references therein.
It is interesting to notice that the nonlinear differential operator in the left hand side of \eqref{eq:BI2} naturally appears  in other contexts, as for instance in the nonlinear electrodynamics of Born and Infeld 
\cite{B1,B2, BI,BI2,Schr} and in string theory, see e.g. \cite{Gibb98}. Calabi \cite{CE} showed in 1968 that Equation \eqref{eq:BI2} with $\rho=0$ has the Bernstein property in dimension $N\le 4$, that is any solution $u$ has to be affine. 
Cheng and Yau \cite{CY} completed this so-called Calabi-Bernstein's theorem and proved in fact its parametric version, i.e. in every dimension, the only maximal space-like hypersurface which is a closed subset of the Lorentz-Minkowski space is a linear hyperplane. The validity of the Bernstein property in every dimension contrasts with 
the less rigid euclidian case \cite{Bern,Sim,Bom,Alm,Flem,SSY}. The generalization to entire area maximizing hypersurfaces is attributed to Bartnik, see \cite[Theorem F]{Eck}. 
Later Treibergs  \cite{T82} tackled the case of entire spacelike hypersurfaces of constant mean curvature. Extensions in Lorentzian product spaces and generalized Robertson-Walker spacetimes were also considered, see for instance \cite{Al,Alb,Pel} and the citations therein. 

The Dirichlet problem for spacelike vertical graphs in $\L^{N+1}$ was solved by Bartnik and Simon in their seminal paper \cite{BS}, and Gerhardt \cite{GH} extended those results to the case of vertical graphs contained in a Lorentzian manifolds which can be expressed as a product of a Riemannian manifold times an interval. Solutions with singularities have been considered in \cite{ Eck,Kl1, Kl2, KM}.

Bayard \cite{Bayard} studied the more general problem of prescribed scalar curvature. For radial graphs, to our knowledge, the only available result concerns entire spacelike hypersurfaces with prescribed scalar curvature which are asymptotic to the light-cone \cite{Bayard1} and our contribution \cite{BIA2} where we investigate the existence and uniqueness of spacelike radial graphs spanning a given boundary datum lying on the hyperbolic space $\mathbb H^N$.

Unbounded data $\rho$ were first considered in \cite{BDP}, but only in the radial case, and in \cite{BIA, Haa}, where the regularity of the solution was the object of attention. 

For weak solutions to the $p$-Laplacian equation with finite mass Borel measure data $\mu$, 
$$-{\rm div}(|\nabla u|^{p-2}\nabla u)=\mu \quad\hbox{in }\mathbb{R}^N,$$
with $p \geq 2$, 
Kuusi and Mingione proved in \cite{KUM2013} the pointwise gradient estimate
$$|\nabla u(x)|^{p-1} \le c(N, p) {\mathrm{I}}^{|\mu|}_1(x,\infty),$$
where
\begin{equation*}
{\mathrm{I}}^{|\mu|}_1(x,\infty):=\int_0^{+\infty}\frac{|\mu|(B_t(x))}{t^N}dt 
\end{equation*}
is the linear Riesz potential associated with $|\mu|$. In particular, if $\mu$ is an $L^q$-datum with $q<N$ which is also locally $L^s$ for some $s>N$, then the solution has a bounded gradient.

More in general, as explained in \cite{BIA}, solutions of the regular quasilinear, possibly degenerate, equation of the form $$
-{\rm div}\left(\frac{g(|\nabla u|)}{|\nabla u|}\nabla u\right) =\mu
$$
in a bounded domain (see \cite{BA} for the assumptions on $g$) satisfy the estimate
\begin{equation}\label{potentialest}
g(|\nabla u(x)|)\le c {\mathrm{I}}_1^{|\mu|}(x,2R)+c g\left(\Mint_{B_R(x)}|\nabla u|dy\right) \,,
\end{equation}
where $c>0$ depends only on $N$
and ${\mathrm{I}}_1^{|\mu|}(x,2R)$ is the truncated linear Riesz potential (see \cite[Section 1]{BA}, or \cite{KUM}). We also refer to  \cite{KUM2013, KUM2014,Ming11,DM, KUM}  for similar estimates for equations driven by the $p$-Laplacian (or more general operators). If this formal inversion of the divergence operator could be justified also for the singular operator appearing in \eqref{eq:BI2}, namely
\[\mathcal{Q}^-(u):=-\operatorname{div}\left(\displaystyle\frac{\nabla u}{\sqrt{1-|\nabla u|^2}}\right),\]
one would derive the estimate 
\begin{equation}\label{eq1:intro}
\frac{|\nabla u(x)|}{\sqrt{1-|\nabla u(x)|^2}} \le c \int_0^\infty\frac{\int_{B_t(x)}|\rho(y)|\, dy}{t^N}dt,
\end{equation}
and therefore the norm of the gradient of the solution would stay away from $1$ as soon as $\rho\in L^s(\R^N)$ with $s<N$ and $\rho$ is locally $L^q(\R^N)$ for some $q>N$. A simple explicit example shows that the condition $\rho\in L_{loc}^q(\R^N)$ for some $q>N$ is sharp, see \cite{BIA}.
It seems quite hard and unlikely to obtain the estimate \eqref{eq1:intro}, at least in a direct way. 
Indeed the mean curvature operator $\mathcal{Q}^-$
does not satisfy the growth and ellipticity conditions appearing in the so-called Nonlinear Calder\'on-Zygmund theory, see \cite[(1.2)]{DM, KUM} or \cite[(4.9)]{MING}. Roughly speaking, the Nonlinear Calder\'on-Zygmund theory is modelled around the $p$-Laplacian as reference operator, which is singular or degenerate as the gradient of $u$ is zero, a point, but $\mathcal{Q}^-$ is singular as the gradient lies on the unit sphere, so from the structural point of view they are clearly distinct operators.

Our main result is a sharp gradient estimate for smooth classical solutions of \eqref{eq:BI2}, see Theorem \ref{teo:mainestimate}, which extends and refines Bartnik and Simon estimates, see \cite{BS,BIA}. Among the several possible applications, this gradient estimate is then used to deduce $W^{2,q}_{\rm loc}$ regularity and actually  solve \cite[Conjecture 1.4 ]{BIA}, with an explicit H\"older exponent (see Theorem \ref{teo:w2qregmin} and Corollary \ref{cor:teo:w2qregmin}), as further described in this section.\\

We adopt the following notations and definitions:
\begin{itemize}
\item $x=(x_1,\ldots,x_N)$ denotes a point in $\RN$, while $\mathbf{x}=(x_1,\ldots,x_{N+1})$ denotes a point in $\LN$;
\item $f_i$ or $\frac{\partial f}{\partial x_i}$ denote the partial derivatives of $f:\R^N\to \R^k$, $N,k\in\N^+$, and if $N=1$ we simply use $D_t f$ or $f^\prime$ to denote the standard derivative;
\item $\nabla f$, $D^2 f$ denote, respectively, the gradient and the Hessian matrix of $f:\RN\to \R$;
\item $(\cdot,\cdot)_{\RN}$ denotes the Euclidean scalar product in $\RN$; 
\item $|\cdot|$ denotes the Euclidean norm of a vector or a matrix;
\item $B_R(x_0):=\{x\in \R^N;  \ |x-x_0|<R\}$ denotes the Euclidean ball of radius $R$ centred at $x_0$;
\item  $u\in C^{0,1}(\RN)$ is said to be:
	\begin{itemize}
		\item {\em weakly spacelike} if $|\n u|\le 1$ a.e. in $\RN$;
		\item {\em spacelike} if $|u(x)-u(y)|<|x-y|$ whenever $x,y\in\RN$, $x\neq y$;
		\item {\em strictly spacelike} if $u\in C^1(\RN)$ and $|\n u|< 1$ in $\RN$;
\end{itemize}
\item given a strictly spacelike function $u\in C^1(\RN)$, we use the notation $v$ to denote the function
\begin{equation}\label{eq:v}
v(x):=\sqrt{1-|\nabla u(x)|^2} \  \ \ x\in \RN;
\end{equation}
\item  the projection of the Lorentz ball of radius $R$ centred at $x_0\in\RN$, associated to a spacelike function $u \in C^{0,1}(\RN)$, is denoted by
$$K_R(x_0):=\{x \in \R^N; \ [|x-x_0|^2-(u(x)-u(x_0))^2]^{1/2}<R \};$$
\item $\omega_N$ denotes the volume of the Euclidean unit ball in $\R^N$;
\item $2_*$ denotes the conjugate H\"older exponent of the critical Sobolev exponent, namely $2_*:=(2^*)^\prime=\frac{2N}{N+2}$, where $2^*:=\frac{2N}{N-2}$, $N\geq 3$;
\item $|\cdot|_q$ denotes the standard norm in $L^q(\RN)$, $q\geq 1$ or $q=\infty$, and $|\cdot|_{q,E}$ denotes the standard norm in $L^q(E)$, where $E\subset \RN$ is a Lebesgue measurable set;
\item $D^{1,2}(\RN)$ is the completion of $C_c^\infty (\RN)$ with respect to the $L^2$ norm of the gradient;
\item $\X$ is the convex set defined by
\begin{equation}\label{eq:defsetX}
\X:=D^{1,2}(\RN)\cap \left\{u\in C^{0,1}(\RN) ;\ |\nabla u|_\infty \le 1\right\},
\end{equation}
and equipped with the $D^{1,2}(\RN)$ norm;
\item $\X^*$ denotes the dual of $\X$ and $\langle \cdot, \cdot \rangle$ denotes the duality pairing between $\X^*$ and $\X$. 
\end{itemize}

\begin{theorem}\label{teo:mainestimate}
 Assume $N\geq 3$, $x_0 \in \RN$, $R>0$, $\rho \in C^1(\RN)$, $q>N$ and $\beta:=\frac{2N}{q}$.
Let $u \in C^3(\RN)$ be a strictly spacelike classical solution of \eqref{eq:BI2} and suppose that $K_R(x_0)$ is bounded. Then there exist $\gamma \in \left(0,\frac{1}{N}\right)$ and $c>0$, both depending on $N$ only, such that 
 \begin{eqnarray}\label{eqTesiProp2}
\displaystyle  \omega_Nv^\gamma(x_0)\!\!\!&\geq &\!\!\! \displaystyle R^{-N} \int_{K_{R}(x_0)} v^{\gamma+1} \ dx + \displaystyle c R^{2-N}\int_{K_{R/2}(x_0)}v^{\gamma-1}\left(|D^2u|^2 + |\nabla v|^2\right) \ dx \nonumber\\
 && \displaystyle -  \frac{3}{2}|\rho|_{q,K_R(x_0)}^2 \int_0^R  s^{1-\beta}  \left(\omega_N^{\frac{1}{q}} + \frac{3\ \omega_N^{-\frac{1}{q}} |\rho|_{q,K_R(x_0)}^2}{2q(2-\beta)}\ s^{2-\beta} +\frac{|\rho|_{q,K_R(x_0)}}{2q(1-\frac{\beta}{2})}\ s^{1-\frac{\beta}{2}} \right)^{q-2} ds\nonumber \\
 &&\displaystyle-  \frac{1}{2}|\rho|_{q,K_R(x_0)} \int_0^R  s^{-\frac{\beta}{2}}  \left(\omega_N^{\frac{1}{q}} + \frac{3\ \omega_N^{-\frac{1}{q}} |\rho|_{q,K_R(x_0)}^2}{2q(2-\beta)}\ s^{2-\beta} +\frac{|\rho|_{q,K_R(x_0)}}{2q(1-\frac{\beta}{2})}\ s^{1-\frac{\beta}{2}} \right)^{q-1} ds.
\end{eqnarray}
\end{theorem}
The gradient estimate in Theorem \ref{teo:mainestimate} is sharp for two reasons. First, in contrast to what we proved in \cite{BIA}, we are able to reach the sharp regularity threshold $q>N$ and we do not assume any global integrability assumption on $\rho$. In particular Theorem \ref{teo:mainestimate} can be applied to entire constant or asymptotically constant mean curvature hypersurfaces, provided that $K_R(x_0)$ is a bounded subset of $\RN$. This assumption plays a crucial role in the proof of Theorem \ref{teo:mainestimate} as we consider integral quantities over the Lorentz ball involving $v^\gamma$ and $\rho$ (see \eqref{monotform}). However, when $u\in L^\infty(\RN)$ then $K_R(x_0)$ is automatically bounded, for any choice of $x_0$, $R$, as shown in Lemma \ref{lem:boundedLorentzballX}-(i). The same conclusion holds true even when $u$ is not globally bounded and possibly asymptotic to the light-cone at infinity, i.e. such that $|\nabla u(x)|\to 1$ as $|x|\to +\infty$, provided that $R$ is sufficiently small, see Lemma \ref{lem:boundedLorentzballX}-(ii). We also stress that the hypothesis $u \in C^3(\RN)$ can be relaxed, but for our purposes (see the proof of Theorem \ref{teo:w2qregmin}) it does not change the analysis since we use it for smooth approximating solutions as in \cite{BS, BIA}.

Next, when considering  the solution of \eqref{eq:BI2} vanishing at infinity, i.e. the unique weak solution of the Born-Infeld equation 
\begin{equation}\label{eq:BI}
\tag{$\mathcal{BI}$}
-\operatorname{div}\left(\displaystyle\frac{\nabla u}{\sqrt{1-|\nabla u|^2}}\right)= \rho \quad \hbox{in }\mathbb{R}^N,
\quad \displaystyle\lim_{|x|\to \infty}u(x)= 0, 
\end{equation} 
we deduce the following result. Before stating the proposition, we first recall the definition of weak solution to \eqref{eq:BI}. 
\begin{definition}\label{def:weaksol} We say that $u\in \X$ is  a {\em weak solution} of \eqref{eq:BI} if for all $\psi \in\X$ we have
\begin{equation}\label{eq:weakBI}
\irn \frac{\n u \cdot \n \psi}{\sqrt{1-|\nabla u|^2}}\, dx
=\langle \rho, \psi\rangle. 
\end{equation}
 \end{definition}

\begin{proposition}\label{prop:mainestimate0-intro}
Assume $N\geq 3$, $q>N$ and let $u \in \X\cap C^3(\RN)$ be a strictly spacelike weak solution of \eqref{eq:BI} with $\rho \in L^m(\RN) \cap C^1(\RN)$, $m\in(1,2_*]$ (the case $m=1$ requires a separate statement, see Proposition \ref{prop:mainestimate20}). There exist two positive constants $\gamma \in (0,\frac{1}{N})$ depending only on $N$, $c$  depending only on $N$ and $m$, such that for any $R>0$, $x_0\in \RN$
  \begin{equation}\label{eq:introgradestim}
\displaystyle (1-|\nabla u|^2)^{\frac{\gamma}{2}}(x_0)\geq  \displaystyle \frac{\omega_N^{\gamma+1}}{\left(\omega_N+ c R^{-N}|\rho|_m^{\frac{Nm}{N-m}}\right)^{\gamma+1}} - P\left(|\rho|_{q,K_R(x_0)}R^{\frac{q-N}{q}}\right),
\end{equation}
where $P:[0,+\infty) \to [0,+\infty)$ is a continuous function such that $P(0)=0$, see \eqref{eq:defPk} for the precise definition.
\end{proposition}

When $\rho=0$, and since $P(0)=0$, we infer that
$$
 \inf_{\RN}(1-|\nabla u|^2)^{\frac{\gamma}{2}} \geq  1,
$$
which shows that $\sup_{\RN}|\nabla u|=0$, i.e. $u=0$ due to the condition at infinity.\\ 

The proof of Theorem \ref{teo:mainestimate} takes its roots from the methods introduced by Bartnik and Simon to deal with bounded data (see \cite[Lemma 2.1]{BS}). The reason of the optimality of \eqref{eqTesiProp2} is that, on the contrary to what is proved in \cite{BS,BIA}, and thanks to a finer decomposition of integral quantities, we derive a monotonicity formula, namely \eqref{eq5monform}, which involves only nonlinear powers of the auxiliary function
$$\psi(s):=
s^{-N} \int_{K_{s}(x_0)} v^{\gamma+1} \ dx.$$
In particular, instead of using an integrating factor to rule out the linear term in $\psi$, as done in  \cite{BS,BIA}, we prove a specific variant of Gronwall's Lemma involving two separate powers of $\psi$ and with two different singular weights (see Lemma \ref{lem:genGronwall}), so that, at the end, there are no extra factors in front of the first integral in \eqref{eqTesiProp2}.
The quantity $v$ appearing in Theorem \ref{teo:mainestimate} has a significant geometrical meaning as it represents the area element of $M=graph(u)$ (see Sect. 2), while the terms $|D^2u|^2$, $|\nabla v|^2$ are related to the second fundamental form of $M$ (see \eqref{normSecForm}, \eqref{eq:NormSecForm2}).\\

Finally, with \eqref{eqTesiProp2} at hand, we are able to improve the results of \cite{BIA, Haa} concerning the regularity of the minimizer of the energy associated with \eqref{eq:BI}, i.e. the Born-Infeld electrostatic energy. We briefly summarize the known facts about the variational formulation of \eqref{eq:BI}. For a given $\rho \in \X^*$, equation \eqref{eq:BI} is, at least formally, the Euler-Lagrange equation associated to the functional $I_\rho:\X\to \R$ defined by
\beq\label{eq:BIenergygeneral}
I_\rho(u)=\irn \Big(1 - \sqrt{1-|\nabla u|^2}\Big) dx- \langle \rho, u\rangle.
\eeq
We recall that $\X^*$ contains Radon measures as for instance linear combinations of Dirac deltas or $L^m(\RN)$ functions, for any $m\in[1,2_*]$, see \cite{FOP,BDP,K} and Remark \ref{rem2}. 
The following properties hold true (see \cite[Lemma 2.1, Lemma 2.2]{BDP}):
\begin{itemize}
\item for any $u\in \X$ one has $\lim_{|x|\to\infty}u(x)=0$;
\item  $\X$ is weakly closed, it embeds continuously into $W^{1,s}(\RN)$ for all $s\in[2^*,\infty)$ and in $L^\infty(\RN)$;
\item $I_\rho$ is  bounded from below in $\X$, coercive, weakly lower semi-continuous and strictly convex.
\end{itemize}
In particular, from the direct methods in the Calculus of Variations the functional $I_\rho$ has a unique minimizer $u_\rho\in \X$ (see \cite[Proposition 2.3]{BDP}) and the set of singular points $$E_{\rho}=\{x\in \RN; \ |\nabla u_\rho|=1\}$$ is a null set with respect to the Lebesgue measure (see \cite[Proposition 2.7]{BDP}).
In general one has $E_\rho \neq \emptyset$, as it happens for instance when $\rho$ is a Dirac mass \cite[Theorem 1.6]{BDP} and \cite{BI} or 
$\rho$ is, in a ball $B_R(0)$, the toy radial datum $1/|x|^{\beta+1}$  with $\beta>0$ and small \cite[Section 1]{BIA}. Hence, as the functional $I_\rho$ is not $C^1$ at points $u\in \X$ such that $|\nabla u|_\infty=1$, we cannot say in general that the unique minimizer is a weak solution to \eqref{eq:BI}.\\

\noindent As pointed out in \cite[Section 2]{BDP}, if $\rho$ is a distribution, the weak formulation of \eqref{eq:weakBI} extends to any test function $\psi \in C^{\infty}_{c}(\RN)$. Moreover, any weak solution of \eqref{eq:BI} in $\X$ coincides with the minimizer of  $I_\rho$ \cite[Proposition 2.6]{BDP}, but, as pointed out before, we do not known if the converse holds true in general. This remains a challenging open question for generic data $\rho \in \X^*$, and positive answers have been given only for bounded or integrable distributions $\rho$, see \cite{BS,BDP,BIA,Haa}.

In view of the heuristic \eqref{eq1:intro}, we conjectured \cite[Conjecture 1.4]{BIA} that if $\rho \in \X^*$ and $\rho \in L^q_{loc}(\RN)$, with $q>N$, then the unique minimizer $u_\rho$ is $C^{1,\alpha}_{loc}(\RN)$, for some $\alpha \in (0,1)$. Under a smallness assumption on $\rho$, we proved $C^{1,\alpha}_{loc}$ regularity for $\rho \in L^q(\RN)$ with $q>2N$ only, see \cite[Theorem 1.6]{BIA}. With \eqref{eq:introgradestim} (and the subsequent Proposition \ref{prop:mainestimate1}) at hand, the argument in \cite[Theorem 1.6]{BIA} shows
there exists a positive constant $c=c(N,q,m)$ such that if $\rho\in L^q(\mathbb{R}^N)\cap L^{m}(\R^N)$ satisfies  
\begin{equation}\label{smallness}
|\rho|_{q}^{\frac{q}{q-N}}|\rho|_{m}^{\frac{m}{N-m}} \leq c, 
\end{equation}
with $q>N$ and $m \in (1,2_*]$, 
then the unique minimizer $u_\rho$ of \eqref{eq:BIenergygeneral} is a weak solution of \eqref{eq:BI}, it is strictly spacelike and $u_\rho \in C^{1,\alpha}_{\rm loc}(\R^N)$ for some $\alpha \in (0,1)$. Differently from  \cite[Theorem 1.6]{BIA} the smallness condition appearing in \eqref{smallness} is intrinsic because it is invariant under a natural transformation associated with \eqref{eq:BI} which preserves the modulus of the gradient (see Remark \ref{rem1}). Recently Haarala posted a preprint \cite{Haa} where he claims, see \cite[Theorem 1.3]{Haa}, $C^{1,\alpha}$ regularity without any smallness assumptions on $\rho$. The strategy of \cite{Haa} is based on a tricky gradient estimate, see \cite[Theorem 3.5]{Haa}, in the spirit of \cite[Theorem 3.5]{BS}, and a delicate fixed point method. With the combination of Theorem \ref{teo:mainestimate} and \cite[Theorem 3.5]{Haa}, it is in fact possible to prove more.

Indeed, by the same proof of \cite[Theorem 1.5]{BIA} and exploiting Theorem \ref{teo:mainestimate} one gets immediately that the minimizer $u_\rho$ has a second weak derivative locally, whenever $\rho \in L^q_{loc}(\RN)$, $q>N$, which is another serious issue. But this is not the end, in fact, notice that in \eqref{eqTesiProp2} there is a weight $v^{\gamma-1}$ in front of the terms $|D^2u|^2$ and $|\nabla v|^2$, and the exponent $\gamma-1$ is negative, because $\gamma \in (0,\frac{1}{N})$. This lead us to suspect that actually the minimizer is more than $W^{2,2}$ locally and this is our second main result.

\begin{theorem}\label{teo:w2qregmin}
Assume $N\geq 3$, $q>N$, $m\in[1,2_*]$, $\rho\in L^q(\mathbb{R}^N) \cap L^m(\RN)$ and let $u_\rho\in \X$ be the unique minimizer of \eqref{eq:BIenergygeneral}. Then $u_\rho \in W^{2,q}_{\rm loc}(\R^N)$, $u_\rho$ is strictly spacelike and it is a weak solution of \eqref{eq:BI}.
\end{theorem}

We then deduce the answer to \cite[Conjecture 1.4]{BIA} when $\rho\in L^q(\mathbb{R}^N)$, with $q>N$, with the  precise value of the H\"older exponent $\alpha=1-\frac{N}{q}$. The case $\rho\in L^q_{\rm loc}(\mathbb{R}^N)$ remains open.
\begin{corollary}\label{cor:teo:w2qregmin}
Assume $N\geq 3$, $q>N$, $m\in[1,2_*]$, $\rho\in L^q(\mathbb{R}^N) \cap L^m(\RN)$. Then the unique minimizer $u_\rho$ of \eqref{eq:BIenergygeneral} belongs to $C^{1,\alpha}_{loc}(\RN)$, with $\alpha=1-\frac{N}{q}$.
\end{corollary}

Our strategy is to derive a uniform control of the gradient at infinity from 
Proposition \ref{prop:mainestimate0-intro} for a sequence of smooth solutions corresponding to a smoothing of the datum $\rho$. In this way we somehow recover the boundary barrier construction performed by Bartnik and Simon in \cite{BS}, which was crucially based on the boundedness of the mean curvature. Then, using Haarala's gradient estimate, namely Theorem \ref{teo:maingradestimate2}, we capitalize on the control at infinity to deduce a uniform global gradient bound. We conclude by using standard elliptic regularity theory for equations in non-divergence form.

The outline of the paper is the following. In Section 2 we recall the basic facts about the geometry of spacelike hypersurfaces in the Lorentz-Minkowski space. Section 3 contains the proof of Theorem \ref{teo:mainestimate}. 
In section 4 we detail the proof of the gradient estimate due to Haarala \cite[Theorem 3.5]{Haa}. In Section 5 we prove our new gradient estimates for the Born-Infeld equation, namely Proposition \ref{prop:mainestimate0-intro} and a similar statement when $m=1$. Finally Section 6 is dedicated to the proof of Theorem \ref{teo:w2qregmin}.

\section{Preliminaries on the differential geometry of hypersurfaces in $\L^{N+1}$}

In this section we collect, for the reader's convenience, some definitions and results about the geometry of spacelike hypersurfaces in the Lorentz-Minkowski space that will be used throughout the paper. 
These facts are essentially contained in \cite{BS, BDP, BIA, BIA2, Lopez2014}.\\

We denote by $\L^{N+1}$ the $(N+1)$-dimensional Lorentz-Minkowski space, which is defined as the vector space $\R^{N+1}$ equipped with the symmetric bilinear form $(\cdot,\cdot)_{\L^{N+1}}$ defined by \eqref{def:Lorentflatzmetric}.
The bilinear form $(\cdot,\cdot)_{\L^{N+1}}$ is non-degenerate and it has index one (see e.g. \cite{Spivak}). 
The modulus of $\mathbf{x} \in \L^{N+1}$ is given by $\|\mathbf{x}\|_{\L^{N+1}}:=|(\mathbf{x},\mathbf{x})|_{\L^{N+1}}^{1/2}$.
We say that a vector $\mathbf{x} \in \L^{N+1}$ is 
\begin{itemize}
\item spacelike if $(\mathbf{x},\mathbf{x})_{\L^{N+1}}>0$ or $\mathbf{x}=0$;
\item timelike if $(\mathbf{x},\mathbf{x})_{\L^{N+1}}<0$;
\item lightlike if $(\mathbf{x},\mathbf{x})_{\L^{N+1}}=0$ and $\mathbf{x}\neq 0$.
\end{itemize}
If $V$ is a vector subspace of $\L^{N+1}$ we define the induced metric $(\cdot,\cdot)_V$ in the standard way
$$(\mathbf{x},\mathbf{y})_V:=(\mathbf{x},\mathbf{y})_{\L^{N+1}}, \ \ \mathbf{x},\mathbf{y} \in V.$$

\begin{definition}
We say that a hypersurface $M\subset \L^{N+1}$ is spacelike (resp. timelike, lightlike) if for any $\mathbf{p} \in M$ the induced metric $(\cdot,\cdot)_{T_{\mathbf{p}}M}$ on the tangent space $T_{\mathbf{p}}M$ is positive definite (resp. $(\cdot,\cdot)_{T_{\mathbf{p}}M}$ has index one, $(\cdot,\cdot)_{T_{\mathbf{p}}M}$ is degenerate). We say that $M$ is a non-degenerate hypersurface if $M$ is spacelike or timelike.
\end{definition}

Notice that if $M$ is a spacelike hypersurface then it inherits a Riemannian structure in a natural way. Moreover observe that if $M$ is a spacelike (resp. timelike) hypersurface and ${\mathbf{p}}\in M$ then we can decompose the space as $\L^{N+1}=T_{\mathbf{p}} M \oplus (T_{\mathbf{p}} M)^\perp$, where $(T_{\mathbf{p}} M)^\perp$ is a timelike (resp. spacelike) subspace of dimension 1 (see \cite{Lopez2014}).  
\begin{definition}
Let $M$ be a non-degenerate hypersurface. A Gauss map is a differentiable map $\bm{\nu}:M \to \L^{N+1}$ such that $\|\bm{\nu}({\mathbf{p}})\|_{\L^{N+1}}=1$ and $\bm{\nu}({\mathbf{p}}) \in (T_{\mathbf{p}} M)^\perp$ for all ${\mathbf{p}} \in M$.
\end{definition}

For the sake of completeness we recall that closed hypersurfaces (i.e. compact hypersurfaces without boundary) do not play a relevant role in the geometry of $\LN$. Indeed the following result holds (see \cite[Proposition 3.1]{Lopez2014} or \cite[Proposition 2.5]{BIA2})
\begin{proposition}\label{prop:closed}
Let $M \subset \L^{N+1}$ be a compact spacelike, timelike or lightlike hypersurface. Then $\partial M \neq \emptyset$.
\end{proposition}

From now we focus only on the case of spacelike vertical graphs. Let $\{\mathbf{e_1},\ldots, \mathbf{e_{N+1}}\}$ be the standard basis of $\LN$ (i.e. the standard basis of $\R^{N+1}$). According to the notations of \cite{BS}, we agree that the indices $i,j$ have the range $1,\ldots,N$, while the indices $\mathcal{I},\mathcal{J}$ have the range $1,\ldots,N+1$, and we observe that $(\mathbf{e_{\mathcal{I}}}, \mathbf{e_\mathcal{J}})_{\LN}=0$ if $\mathcal{I}\neq \mathcal{J}$, $(\mathbf{e_{i}},\mathbf{e_{i}})_{\L^{N+1}}=1$ and  $(\mathbf{e_{N+1}},\mathbf{e_{N+1}})_{\L^{N+1}}=-1$.
\begin{definition}
We say that a timelike vector $\mathbf{x}\in\LN$ is future-directed (resp. past-directed) if $(\mathbf{x},\mathbf{e_{N+1}})_{\LN}<0$ (resp. $(\mathbf{x},\mathbf{e_{N+1}})_{\LN}>0$).
\end{definition}

Let $\Omega\subset \RN$ be a domain (bounded or unbounded) and let $u\in C^1(\Omega)$, the associated vertical graph is
$$M=\{(x,u(x)) \in \LN; \ x\in \Omega\}.$$
An obvious parametrization is given by $\mathbf{\Phi}:\Omega\to\LN$, $\mathbf{\Phi}(x)=(x,u(x))$ and thus
\beq\label{eq:BaseTpM}
\mathbf{X_i}:=\frac{\partial \mathbf{\Phi}}{\partial x_i}=\mathbf{e_i}+u_i\mathbf{e_{N+1}}, \ \ i=1,\ldots,N,
\eeq
 is a basis of tangent vectors for $T_{{\mathbf{p}}}M$, where ${\mathbf{p}}=(x,u(x))$. The induced metric on $M$ is given by $g=(g_{ij})_{i,j=1,\ldots,N}$, where $$g_{ij}=(\mathbf{X_i}, \mathbf{X_j})_{\L^{N+1}}=\delta_{ij}-u_iu_j.$$
Since the determinant of the principal minor of order $k$ of the matrix $(g_{ij})_{i,j=1,\ldots,N}$, denoted by $g_k$, is $Det(g_k)=1-\sum_{i=1}^ku_i^2$,  it is clear that $M$ is spacelike (resp. timelike, lightlike) if and only if $|\nabla u|<1$ (reps. timelike if and only if $|\nabla u|>1$,  lightlike if and only if $|\nabla u|=1$), for all $x\in \Omega$. 
Accordingly, we say that $u\in C^1(\Omega)$ is strictly spacelike if $|\nabla u|<1$.

Let $u\in C^2(\Omega)$ be a strictly spacelike function and let $M$ be the associated cartesian graph. The future-directed Gauss map is expressed by
$$\bm{\nu}(x)=\frac{(\nabla u(x), 1)}{\sqrt{1-|\nabla u(x)|^2}},$$
and we denote by $\nu_1,\ldots, \nu_{N+1}$ its components.
Recalling the notation $v=\sqrt{1-|\nabla u|^2}$ introduced in \eqref{eq:v}, we can then write
\beq\label{eq:Gausscomponents}
\bm{\nu}=\sum_{i=1}^N \nu_i \mathbf{e_i} + \nu_{N+1} \mathbf{e_{N+1}} = \sum_{i=1}^N \frac{u_i}{v} \mathbf{e_i} + \frac{1}{v} \mathbf{e_{N+1}}.
\eeq
The coefficients and the norm of the second fundamental form $\mathrm{II}$ of $M$ are given, respectively, by 
$$
{\mathrm{II}}_{ij}=(\mathbf{X_i},\nabla_{\mathbf{X_j}} \bm{\nu})_{\L^{N+1}}=\frac{1} {v}u_{ij}, 
$$
and
\begin{equation}\label{normSecForm} 
 \|{\mathrm{II}} \|^2:=\sum_{i,j,k,l=1}^N g^{ij}g^{kl}{{\mathrm{II}}}_{ik}{{\mathrm{II}}}_{jl}=\frac{1}{v^2}\sum_{i,j,k,l=1}^Ng^{ij}g^{kl}u_{ik}u_{jl},
 \end{equation}
where $\{\mathbf{X_i}\}_{i=1,\ldots,N}$ is the basis of $T_{\mathbf{p}}M$ found in \eqref{eq:BaseTpM}, $g^{-1}=(g^{ij})_{ij=1,\ldots,N}$, $g^{ij}=\delta_{ij}+\nu_i \nu_j$ is the inverse matrix of $g$. Notice that from \eqref{normSecForm}, see also \eqref{eq:finalmainteo}, it follows that
\begin{equation}\label{eq:NormSecForm2}
 \begin{array}{lll}
 \displaystyle  v^2\|{\mathrm{II}} \|^2 &= &  \displaystyle \sum_{i,k=1}^N u_{ik}^2 + 2 \sum_{i=1}^N\left(v^{-1}\sum_{k=1}^Nu_k u_{ik}\right)^2 + v^{-2}\left(v^{-1}\sum_{i,k=1}^N u_i u_k u_{ik} \right)^2\\[16pt]
  &= &  \displaystyle |D^2 u|^2 + 2|\nabla v|^2 + v^{-2} |(\nabla u, \nabla v)_{\RN}|^2. 
 \end{array}
 \end{equation}
The mean curvature of $M$ (see \cite{Lopez2014} for more details) is defined as 
\beq\label{def:meancurv}
H:=-\sum_{i,j=1}^N g^{ij}{{\mathrm{II}}}_{ij}=-\frac{1}{v} \sum_{i,j=1}^N g^{ij} u_{ij}.
\eeq
Notice that for notational convenience, and differently from \cite{Lopez2014} (and other references), we do not divide by $N$ in the definition of $H$. Moreover, when applying the results of \cite{BS} we stress that $H$ has the opposite sign.

We recall now some explicit formulas for the differential operators on a spacelike vertical graph $M$ (we refer to \cite[Section 2]{BS} for more details). Let $\bm{\delta}=\operatorname{grad}_M$, $\operatorname{div}_M$, $\Delta_M$ be, respectively, the gradient, the divergence and the Laplace-Beltrami operators on $M$. Assume $W \subset \L^{N+1}$ is an open neighborhood of $M$, $\mathbf{Y}$ is a $C^1$ vector field on $W $, and $f \in C^1(W)$ is such that $\frac{\partial f}{\partial x_{N+1}}=0$. We have
\beq\label{eq:Usefulform}
\begin{array}{lll}
\displaystyle \bm{\delta} f = \sum_{\mathcal{I}=1}^{N+1} (\delta_{\mathcal{I}}f)\mathbf{e_{\mathcal{I}}},\ \ \hbox{where}\  \delta_i f=\sum_{j=1}^N  g^{ij} \frac{\partial f}{\partial x_j},\ \delta_{N+1} f=\frac{1}{v}\sum_{i=1}^N  \nu_i \frac{\partial f}{\partial x_i},&&\\[6pt]
\displaystyle \|\bm{\delta} f\|_{\LN}^2 = \sum_{i,j=1}^{N+1} g^{ij}  \frac{\partial f}{\partial x_i} \frac{\partial f}{\partial x_j} = |\nabla f|^2 + \left( \sum_{i=1}^N\nu_i  \frac{\partial f}{\partial x_i}\right)^2,&&\\[16pt]
\displaystyle \operatorname{div}_M \mathbf{Y} = \sum_{i,j=1}^N g^{ij} (\mathbf{X_i}, \nabla_{\mathbf{X_j}} \mathbf{Y})_{\LN}.&&\\
\end{array}
\eeq
Moreover, if in addition $f \in C^2(W)$ then, recalling \eqref{def:meancurv}, we have
\beq\label{eq:UsefulformBis}
\displaystyle \Delta_M f =  \operatorname{div}_M \left( \operatorname{grad}_M f \right)= \sum_{i,j=1}^N g^{ij} \frac{\partial^2 f}{\partial x_i \partial x_j} - \sum_{i=1}^NH \nu_i \frac{\partial f}{\partial x_i} .
\eeq
Another tool that will be used in the next section is the following corollary of Stokes's theorem (for the proof see \cite[Sect. 2]{BS}).

\begin{proposition}[Integration by parts] For any $f \in C_c^1(M)$, $g \in C^1(M)$ one has
\beq\label{eq:integparts}
\begin{array}{lll}
\displaystyle \int_{M} f \delta_{N+1} g\ dA &=& \displaystyle  \int_{M}\frac{1}{v} f g H \ dA - \int_{M} g \delta_{N+1} f\ dA, \\[12pt]
\displaystyle \int_{M} f \delta_{i} g\ dA &=&\displaystyle  \int_{M}\nu_i f g H \ dA-  \int_{M} g \delta_{i} f\ dA , \ \ \ i=1,\ldots,N,
\end{array}
\eeq
where $dA=v dx$ is the induced volume form on $M$, $dx$ being the Lebesgue measure on $\R^N$.
\end{proposition}
We conclude this section by recalling some properties and two important identities involving the gradient and the Laplacian of the Lorentz distance that will play a crucial role in the proof of Theorem \ref{teo:mainestimate}. We begin with a definition.
\begin{definition}\label{defLorentzdist}
Let $u\in C^{0,1}(\RN)$ be a spacelike function and let $x_0\in\RN$. 
We define the Lorentz distance from $(x_0,u(x_0))$ as $$l(x,x_0):=[|x-x_0|^2-(u(x)-u(x_0))^2]^{1/2}.$$ Given $R>0$ and $M=graph(u)$, we define the Lorentz ball of radius $R$ centred at $(x_0,u(x_0))$ as 
$$L_R(x_0):=\{(x,u(x)) \in M; \ l(x,x_0)<R\}$$ 
and its projection on $\RN$ as  $$K_R(x_0):=\{x \in \RN; \ l(x,x_0)<R\}.$$
 \end{definition} 
 For simplicity, when $x_0\in\RN$ is fixed and there is no possible confusion we will adopt the simpler notations $l=l(x)$, $L_R$ and $K_R$.\\ 
 
 \begin{remark}\label{rem:Lorentzball}
From Definition \ref{defLorentzdist} it is clear that the notions of Lorentz distance and Lorentz ball can be extended to the class of weakly spacelike functions. However, observe that if $u\in C^{0,1}(\RN)$ contains a light ray, namely if there exist $x_0, x_1\in \RN$ such that for all $t\geq 0$,
$$|u(x_0+t(x_1-x_0)) -u(x_0)|=t|x_1-x_0|,$$   
then $l(x,x_0)=0$ for any $x=x_0+t(x_1-x_0)$, $t\geq0$, so that $l$ is not a distance in the usual meaning and $K_R(x_0)$ turns out to be an unbounded subset of $\RN$. Nevertheless, if $u$ is weakly spacelike and $u\in L^\infty(\RN)$ then $K_R(x_0)$ is bounded. The same conclusion holds true also for (possibly) unbounded strictly spacelike functions $u$, provided that $R$ is small enough. This is the content of the next lemma.
 \end{remark}
 \begin{lemma}\label{lem:boundedLorentzballX}
 The following assertions hold.
 \begin{itemize}
 \item[i)] Let $u\in C^{0,1}(\RN) \cap  L^\infty(\RN)$ be weakly spacelike. Then, given $R>0$ and $x_0\in\RN$, $K_R(x_0)$ is bounded and $K_R(x_0) \subset B_{R^\prime}(x_0)$, with $R^\prime=\sqrt{R^2+4|u|_\infty^2}$;
\item[ii)] Let $u\in C^1(\RN)$ be strictly spacelike. Then, for any $x_0\in \RN$ there exists $R>0$ such that $K_R(x_0)$ is bounded.
\end{itemize}
 \end{lemma}
 \begin{proof}
We start with assertion i). Take $x_0\in \RN$ and $R>0$. By definition of $K_R(x_0)$, it is clear that
$$|x-x_0|^2<R^2+(u(x)-u(x_0))^2\leq R^2+4|u|_\infty^2,$$ 
for any $x\in K_R(x_0)$, so that i) holds.

In order to prove ii), we fix $x_0\in \RN$ and set
$$\zeta:=\inf_{B_1(x_0)}(1-|\nabla u|).$$
As $u$ is strictly spacelike, we know that $\zeta>0$ and for any $t\in]0,1]$ and any $\omega \in \R^N$ with $|\omega|=1$, we readily check that
\begin{eqnarray*}
\frac{d}{dt}\left[l^2(x_0,x_0+t\omega)\right] &=&2t - 2(u(x_0+t\omega)-u(x_0))(\nabla u(x_0+t\omega) \cdot \omega)\\
 &\geq& 2t - 2(1-\zeta)^2t = 2t(2\zeta-\zeta^2)>0,
\end{eqnarray*}
because $0<\zeta\le 1$.
We then infer that 
$$l^2(x_0,x_0+\omega)=\int_0^1\frac{d}{dt}\left[l^2(x_0,x_0+t\omega)\right] \, dt\geq 2\zeta-\zeta^2,$$
which means that
$$l^2(x_0,x)\geq 2\zeta-\zeta^2, \text{ for any } x\in\partial B_1(x_0).$$ 
We now conclude that $K_R(x_0)\subset B_1(x_0)$ as long as $0<R<\sqrt{2\zeta-\zeta^2}$, i.e. ii) holds.
\end{proof}

Let us fix $x_0\in\RN$, assume that $u\in C^2(\RN)$ is strictly spacelike, set $\mathbf{X_0}:=(x_0,u(x_0))\in \LN$ and let $l$ be the Lorentz distance (from $\mathbf{X_0}$) associated to $u$. We notice that since $l=l(x)$ does not depend on the variable $x_{N+1}$ and $u$ is strictly spacelike then $l$ is naturally defined and of class $C^2$ in a neighborhood $W\subset \LN$ of $M$.  In particular, using  \eqref{eq:Usefulform}-\eqref{eq:UsefulformBis} we can prove that (see \cite[(2.10)]{BS})

\beq\label{eq:Usefulform2a}
\displaystyle \|\bm{\delta} l\|_{\LN}^2 = 1 + l^{-2} (\bm{\nu}, \mathbf{X}-\mathbf{X_0})_{\LN}^2,
\eeq
\beq\label{eq:Usefulform2b}
\displaystyle \Delta_M\left(\frac{1}{2} l^2\right)=N -\left(H\bm{\nu},  \mathbf{X}-\mathbf{X_0}\right)_{\LN}.
\eeq
Finally, we recall the following special case of Federer's coarea formula (see \cite[(2.14)]{BS}, \cite[Theorem 3.2.12]{FED} or \cite[Section 3.4.4]{Evans})
\begin{theorem}
Let $u\in C^1(\RN)$ be a strictly spacelike function, $M=graph(u)$, $x_0\in\RN$, $s>0$ and assume that $K_s(x_0)$ is bounded. For any $h \in C^0(M)$ we have
\beq\label{eqGF2}
D_s \left[ \int_{L_s(x_0)} h \ dA\right] = \int_{\partial L_s(x_0)} h \|\bm{\delta} l\|_{\LN}^{-1} \ d\mu,
\eeq
where $d\mu$ is the surface measure on $\partial L_s(x_0)$.
\end{theorem}

\section{Proof of the gradient estimate \eqref{eqTesiProp2}}

In this section we prove Theorem \ref{teo:mainestimate}. We begin with a preliminary technical result, which is a generalization of the classical Gronwall's lemma.
\begin{lemma}\label{lem:genGronwall}
Let $T>0$, $\beta \in (0,2)$, $q>2$ and assume $\psi:[0,T]\to[0,+\infty[$ is a continuous function such that
\begin{equation}\label{lemtech:hyp}
\psi(t) \leq C_0 + C_1\int_0^t s^{1-\beta} (\psi(s))^{\frac{q-2}{q}} \ ds +  C_2\int_0^t s^{-\frac{\beta}{2}} (\psi(s))^{\frac{q-1}{q}} \ ds \ \ \forall t\in [0,T]
\end{equation}
for some constants $C_0>0$, $C_1,C_2\geq 0$. Then
\begin{equation}
\psi(t) \leq \left(C_0^{\frac{1}{q}} + \frac{C_0^{-\frac{1}{q}}C_1}{q(2-\beta)} t^{2-\beta} +\frac{C_2}{q(1-\frac{\beta}{2})} t^{1-\frac{\beta}{2}} \right)^q \ \ \forall t\in [0,T].
\end{equation}
\end{lemma}

\begin{proof}
Let $T>0$, $\beta\in(0,2)$ and consider the auxiliary function $y:[0,T] \to [0,+\infty[$ defined by
$$y(t):=C_1\int_0^t s^{1-\beta} (\psi(s))^{\frac{q-2}{q}} \ ds +  C_2\int_0^t s^{-\frac{\beta}{2}} (\psi(s))^{\frac{q-1}{q}} \ ds.$$

By definition and thanks to \eqref{lemtech:hyp} we infer that
\begin{eqnarray}\label{eq:propgrom}
\displaystyle y^\prime(t)&=& \displaystyle C_1 t^{1-\beta} (\psi(t))^{\frac{q-2}{q}} +  C_2 t^{-\frac{\beta}{2}} (\psi(t))^{\frac{q-1}{q}}\nonumber \\
&\leq&\displaystyle C_1 t^{1-\beta} (C_0+y(t))^{\frac{q-2}{q}}  +  C_2 t^{-\frac{\beta}{2}} (C_0+y(t))^{\frac{q-1}{q}} \ \  \ \forall t\in(0,T].
\end{eqnarray}
Since $y$ is a non-negative function and $C_0>0$, dividing each side of \eqref{eq:propgrom} by $(C_0+y(t))^{\frac{q-1}{q}}$ we get, for any $t\in(0,T]$,
$$ \frac{y^\prime(t)}{(C_0+y(t))^{\frac{q-1}{q}}} \leq C_1 \frac{t^{1-\beta}} {(C_0+y(t))^{\frac{1}{q}}} +  C_2 t^{-\frac{\beta}{2}} \leq\ C_0^{-\frac{1}{q}}C_1 {t^{1-\beta}} +  C_2 t^{-\frac{\beta}{2}}.$$
Then  integrating on $(0,t)$, taking into account that $\beta\in(0,2)$, we obtain
$$ \int_0^t \frac{y^\prime(s)}{(C_0+y(s))^{\frac{q-1}{q}}} \ ds \leq  \frac{C_0^{-\frac{1}{q}}C_1}{2-\beta} {t^{2-\beta}} +  \frac{C_2}{1-\frac{\beta}{2}} t^{1-\frac{\beta}{2}}  \ \ \forall t \in (0,T].$$
Changing variable in the first integral, we deduce that
\begin{equation}\label{eq:propgrom2}
q(C_0+y(t))^{\frac{1}{q}}-qC_0^{\frac{1}{q}}=  \int_{C_0}^{C_0+y(t)} k^{-\frac{q-1}{q}} \ dk\leq \frac{C_0^{-\frac{1}{q}}C_1}{2-\beta} {t^{2-\beta}} +  \frac{C_2}{1-\frac{\beta}{2}} t^{1-\frac{\beta}{2}} \ \  \forall t \in [0,T],
\end{equation}
since by definition $y(0)=0$.
Finally, recalling that $\psi(t)\leq C_0+y(t)$, from \eqref{eq:propgrom2} we readily obtain
$$\psi(t) \leq  \left(C_0^{\frac{1}{q}} + \frac{C_0^{-\frac{1}{q}}C_1}{q(2-\beta)} t^{2-\beta} +\frac{C_2}{q(1-\frac{\beta}{2})} t^{1-\frac{\beta}{2}} \right)^q \ \ \forall t\in [0,T].$$
\end{proof}

\begin{proof}[Proof of Theorem \ref{teo:mainestimate}]
 Let $N\geq3$, $q>N$, $x_0 \in \R^N$, $R >0$ and assume $\rho$, $u$, $K_R(x_0)$ are as in the statement of Theorem \ref{teo:mainestimate}.
Observe first that since $\hat u(x):=u(x+x_0)-u(x_0)$ is still a solution to \eqref{eq:BI2} with $\rho$ replaced by $\check \rho(x):=\rho(x+x_0)$, and since $\nabla \hat u(0)=\nabla u(x_0)$, $|\rho|_{q,K_R(x_0)}=|\check \rho|_{q,\hat{K}_R(0)}$, where $\hat{K}_R(0)$ is the projection of the Lorentz ball associated to $\hat u$, we can assume without loss of generality that $\mathbf{X_0}:=(x_0,u(x_0))=(0,0)$.

Let $\mathbf{X}:=(x,u(x)) \in \L^{N+1}$ be the position vector, $M=\mathrm{graph}(u)$ and $s \in (0,R]$. Since $u$ is a strictly spacelike classical solution to \eqref{eq:BI2}, $M$ has mean curvature $\rho$. Therefore, as a consequence of Green's formula, Proposition \ref{eq:integparts}, \eqref{eq:Usefulform2a},\eqref{eq:Usefulform2b} and Federer's coarea formula \eqref{eqGF2}, for any $f \in C^2(M)$ we have
\begin{eqnarray}
\label{eq:monotform}
\displaystyle D_s\left[s^{-N}  \int_{ L_s}  f   \ dA\right] &=& \displaystyle \int_{L_s} s^{-N-1} \left(\frac{1}{2}(s^2-l^2) \Delta_M f - f \rho (\mathbf{X},\bm{\nu})_{\LN}\right) \ dA\nonumber \\
 &&\displaystyle  -D_s\left[\int_{ L_s}  f l^{-N-2}(\mathbf{X}, \bm{\nu})_{\LN}^2  \ dA\right], 
\end{eqnarray}
where $L_s$, $K_s$ denote, respectively, the Lorentz ball and its projection on $\RN$ associated to $u$, centred at the origin and of radius $s$ (notice that $L_s$ is a bounded subset of $M$, for any $s\in(0,R]$ because by assumption $K_R$ is bounded). We refer to \cite[(2.15)]{BS} or \cite[(3.4)]{BIA} for the proof of \eqref{eq:monotform}.

Let $\gamma$ be a positive number to be determined later and set $v:=\sqrt{1-|\nabla u|^2}$. Notice that since $u\in C^3(\RN)$ is strictly spacelike then $v^\gamma\in C^2(\RN)$. Moreover, since $v^\gamma$ does not depend on the variable $x_{N+1}$, it is naturally defined and of class $C^2$ in a whole neighborhood of $M$. Exploiting \eqref{eq:Usefulform} and \eqref{eq:UsefulformBis}, we check that
\begin{eqnarray}
\label{eq:monotform2}
\displaystyle \Delta_M v^\gamma \!\!&=& \!\! \displaystyle \gamma v^{\gamma-1} \Delta_M v + \gamma (\gamma - 1) v^{\gamma-2} \|\bm{\delta} v\|_{\LN}^2\nonumber\\
&=&\!\! \displaystyle - \gamma v^{\gamma-2} \left[ \sum_{i,j=1}^N u_{ij}^2 - \gamma \left( \sum_{i=1}^N u_{ii}\right)^2 + (1-\gamma) v \rho \sum_{i=1}^N u_{ii} + v^2\rho^2\!\!+ (1-\gamma) \sum_{j=1}^N\left(\sum_{i=1}^N \nu_i u_{ij}\right)^2\right]\nonumber\\
&& \displaystyle + \gamma \delta_{N+1} \left(v^{\gamma+1}\rho \right).
\end{eqnarray}

By simple algebraic manipulations and using the well known trace inequality 
$$\left(\sum_{i=1}^N u_{ii}\right)^2 \leq N \sum_{i,j=1}^N u_{ij}^2,$$ 
and Young inequality 
\beq\label{eq:elementinterpolineq}
ab\leq \e a^2 + \frac{1}{4\e} b^2  \ \ \forall a,b\geq 0, \ \forall \e>0,
\eeq
 we find two constants $\gamma \in (0,\frac{1}{N})$ and $C>0$, both depending on $N$ only, such that
\beq\label{eq:monotform3}
 \Delta_M v^\gamma \leq - C v^{\gamma-2} \left[\sum_{i,j=1}^N u_{ij}^2 + \sum_{j=1}^N\left(\sum_{i=1}^N \nu_i u_{ij}\right)^2 \right] + \frac{1}{4} v^\gamma \rho^2 + \gamma \delta_{N+1} \left(v^{\gamma+1}\rho \right).
\eeq
Applying \eqref{eq:monotform} with $f=v^\gamma$ and exploiting \eqref{eq:monotform3}, we get
\begin{eqnarray}\label{monotform}
\displaystyle D_s\left\{s^{-N} \int_{L_s} v^\gamma \ dA\right\} 
&\leq&  -C \int_{L_s} \frac{1}{2} s^{-N-1} (s^2-l^2)\left(\sum_{i,j=1}^N u_{ij}^2 + \sum_{j=1}^N\left(\sum_{i=1}^N \nu_i u_{ij}\right)^2\right) v^{\gamma-2} \ dA\nonumber\\
&&+\underbrace{ \int_{L_s} \frac{1}{2} s^{-N-1} (s^2-l^2)\left(\frac{1}{4} v^\gamma \rho^2 + \gamma \delta_{N+1}(v^{\gamma+1}\rho)\right) \ dA}_{(I_1)}\nonumber\\
&&+\underbrace{\int_{L_s}  s^{-N-1} \rho ( \mathbf{X}, \bm{\nu})_{\L^{N+1}} v^{\gamma} \ dA}_{(I_2)} -  \displaystyle  D_s\left\{\int_{L_s} ( \mathbf{X}, \bm{\nu})_{\L^{N+1}}^2 l^{-N-2}v^\gamma \ dA\right\}.
\end{eqnarray}
Let us analyze the terms $(I_1)$ and $(I_2)$. For $(I_1)$, we write
\begin{equation}\label{eq:decI}
(I_1)=\displaystyle \frac{1}{8} s^{-N-1}  \int_{L_s} (s^2-l^2) v^\gamma \rho^2 \ dA + \displaystyle\frac{\gamma}{2}s^{-N-1} \int_{L_s}  (s^2-l^2) \delta_{N+1} (v^{\gamma+1}\rho) \ dA.
\end{equation}
Now, by definition of $L_s$ one has $s^2-l^2\leq 2s^2$ in $L_s$, then, thanks to H\"older's inequality, taking into account that $v\leq 1$, $dA=v dx$ and $\rho \in L_{loc}^q(\RN)$, we have
\begin{eqnarray}\label{eq:estftermI}
\displaystyle  \frac{1}{8} s^{-N-1}  \int_{L_s} (s^2-l^2) v^\gamma \rho^2 \ dA &\leq&\displaystyle   \frac{1}{4} s^{-N+1} \left(\int_{L_s} v^{\gamma\frac{q}{q-2}} \ dA\right)^{\frac{q-2}{q}}\left(\int_{L_s} |\rho|^q \ dA\right)^{\frac{2}{q}}\nonumber\\
&\leq&\displaystyle   \frac{1}{4} s^{-N+1+N\frac{q-2}{q}} \left(s^{-N}\int_{L_s} v^{\gamma} \ dA\right)^{\frac{q-2}{q}}\left(\int_{L_s} |\rho|^q \ dA\right)^{\frac{2}{q}}\nonumber \\
&\leq&\displaystyle   \frac{1}{4}|\rho|_{q,K_s}^2 s^{1-\beta} \left(s^{-N}\int_{L_s} v^{\gamma} \ dA\right)^{\frac{q-2}{q}},
\end{eqnarray}
where $\beta:=\frac{2N}{q}$. For the second integral term in $(I_1)$, using the first relation in \eqref{eq:integparts} with $f=s^2-l^2$ (notice that $f=0$ on $\partial L_s$), $g=v^{\gamma+1}\rho$, $H=\rho$, we deduce that

\begin{eqnarray}\label{eq1monform}
 \displaystyle \frac{\gamma}{2}s^{-N-1}\int_{L_s}  (s^2-l^2) \delta_{N+1}(v^{\gamma+1}\rho) \ dA & = & \underbrace{\displaystyle  \frac{\gamma}{2}s^{-N-1} \int_{L_s} (s^2-l^2)  v^\gamma \rho^2 \ dA}_{(I_3)} \nonumber \\ & &  +\underbrace{\displaystyle \gamma s^{-N-1} \int_{L_s} v^{\gamma+1} \rho l \delta_{N+1}l \ dA}_{(I_4)}.  
\end{eqnarray}
We notice that $(I_3)$ is of the same form as the first term in the right-hand side of \eqref{eq:decI} and thus, arguing as in \eqref{eq:estftermI}, we deduce that
\begin{equation}\label{eq:estIA}
{(I_3)}\leq \gamma |\rho|_{q,K_s}^2 s^{1-\beta} \left(s^{-N}\int_{L_s} v^{\gamma} \ dA\right)^{\frac{q-2}{q}}.
\end{equation}
For $(I_4)$, recalling that $\frac{\partial}{\partial x_{N+1}} l =0$, we deduce from the first relation of \eqref{eq:Usefulform} that 
$$ v \delta_{N+1} l =\sum_{i=1}^N \nu_i \frac{\partial}{\partial x_{i}} l, $$
which in particular implies 
\beq\label{eq:218BS}
|v \delta_{N+1} l | \leq \|\bm{\delta} l \|_{\LN}.
\eeq
Then, using \eqref{eq:218BS}, recalling that $l\leq s$ in $L_s$, exploiting \eqref{eq:Usefulform2a} and the elementary inequality 
\beq\label{eq:elementaryineq2}
(a+b)^\alpha\leq  a^\alpha +  b^\alpha\ \ \forall a,b>0, \forall\alpha\in(0, 1],
\eeq 
with $\alpha=1/2$, we obtain
\begin{eqnarray}\label{eq:decIB}
 \displaystyle{(I_4)}&=&\displaystyle \gamma s^{-N-1} \int_{L_s} v^{\gamma} \rho l v\delta_{N+1}l \ dA\nonumber\\
&\leq&\displaystyle  \gamma s^{-N} \int_{L_s} v^{\gamma} |\rho|  \|\bm{\delta} l \|_{\LN} \ dA\nonumber\\
&\leq&\displaystyle  \gamma s^{-N} \int_{L_s} v^{\gamma} |\rho| \ dA +\gamma s^{-N} \int_{L_s} v^{\gamma} |\rho| l^{-1} |(\mathbf{X},\bm{\nu})_{\LN}| \ dA
\end{eqnarray}
For the first integral in the right-hand side of \eqref{eq:decIB}, applying H\"older's inequality and arguing as in \eqref{eq:estftermI} we readily obtain
\begin{equation}\label{eq:estfirstIB}
 \gamma s^{-N} \int_{L_s} v^{\gamma} |\rho| \ dA \leq \gamma s^{-\frac{\beta}{2}}  \left(s^{-N}\int_{L_s} v^{\gamma} \ dA\right)^{\frac{q-1}{q}}|\rho|_{q,K_s}.
\end{equation}
For the second integral in the right-hand side of \eqref{eq:decIB}, using the elementary inequality 
\begin{equation}\label{elemineq}
|\rho| |( \mathbf{X}, \bm{\nu})_{\L^{N+1}}|  \leq \frac{1}{2}l^2\rho^2+\frac{1}{2}l^{-2} ( \mathbf{X}, \bm{\nu})_{\L^{N+1}}^2,
\end{equation}
taking into account that $l\leq s$ in $L_s$ and \eqref{eq:estIA}, we infer that
\begin{eqnarray}\label{eq:estsecIB}
\displaystyle \gamma s^{-N} \int_{L_s} v^{\gamma} |\rho| l^{-1} |(\mathbf{X},\bm{\nu})_{\LN}| \ dA
&\leq&\displaystyle \frac{\gamma}{2} s^{-N} \int_{L_s} v^{\gamma} l |\rho|^2 \ dA +  \frac{\gamma}{2} s^{-N} \int_{L_s} v^{\gamma} l^{-3}(\mathbf{X},\bm{\nu})_{\LN}^2 \ dA\nonumber\\
&\leq&\displaystyle \frac{\gamma}{2} |\rho|_{q,K_s}^2 s^{1-\beta} \left(s^{-N}\int_{L_s} v^{\gamma} \ dA\right)^{\frac{q-2}{q}} \nonumber\\ 
&& +  \frac{\gamma}{2} s^{-N} \int_{L_s} v^{\gamma} l^{-3}(\mathbf{X},\bm{\nu})_{\LN}^2 \ dA.
\end{eqnarray}
This completes the study of the term $(I_1)$.
For  $(I_2)$, using again \eqref{elemineq} and arguing as in \eqref{eq:estsecIB}, we deduce that
\begin{eqnarray}\label{eq:estII}
 s^{-N-1}\int_{L_s}    \rho ( \mathbf{X}, \bm{\nu})_{\L^{N+1}} v^{\gamma}  \ dA
&\leq&\displaystyle \frac{1}{2} |\rho|_{q,K_s}^2 s^{1-\beta} \left(s^{-N}\int_{L_s} v^{\gamma} \ dA\right)^{\frac{q-2}{q}} \nonumber\\
& & +  \frac{1}{2} s^{-N-1} \int_{L_s}  l^{-2}(\mathbf{X},\bm{\nu})_{\LN}^2 v^{\gamma} \ dA,
\end{eqnarray}

Summing up,  from \eqref{monotform}, taking into account \eqref{eq:decI}--\eqref{eq:estIA}, \eqref{eq:decIB}--\eqref{eq:estsecIB}, \eqref{eq:estII} and recalling that $0<\gamma<\frac{1}{N}$ and $N\geq 3$, we have obtained the estimate
\begin{eqnarray}\label{eq2monform}
\displaystyle D_s\left\{s^{-N} \int_{L_s} v^\gamma \ dA\right\} \!\!\!
&\leq&\!\! \displaystyle -C \int_{L_s} \frac{1}{2} s^{-N-1} (s^2-l^2)\left(\sum_{i,j=1}^N u_{ij}^2 + \sum_{j=1}^N\left(\sum_{i=1}^N \nu_i u_{ij}\right)^2\right) v^{\gamma-2} \ dA\nonumber\\
&&\!\!+  \displaystyle \frac{3}{2}|\rho|_{q,K_s}^2 s^{1-\beta}\left(s^{-N}\int_{L_s} v^{\gamma} \ dA\right)^{\frac{q-2}{q}} + \frac{1}{2} |\rho|_{q,K_s} s^{-\frac{\beta}{2}}\left(s^{-N}\int_{L_s} v^{\gamma} \ dA\right)^{\frac{q-1}{q}}\nonumber\\
&& \!\!+  \underbrace{\frac{1}{4} s^{-N} \int_{L_s} l^{-3} ( \mathbf{X}, \bm{\nu})_{\L^{N+1}}^2  v^{\gamma} \ dA}_{(I_5)} +  \displaystyle  \underbrace{\frac{1}{2} s^{-N-1} \int_{L_s}  l^{-2}(\mathbf{X},\bm{\nu})_{\LN}^2 v^{\gamma}\ dA}_{(I_6)}\nonumber\\
&& \!\!\displaystyle-D_s\left\{\int_{L_s} ( \mathbf{X}, \bm{\nu})_{\L^{N+1}}^2 l^{-N-2}v^\gamma \ dA\right\}.
\end{eqnarray}
We now study the terms $(I_5)$ and $(I_6)$. For  $(I_5)$, by elementary algebraic manipulations and using Federer's coarea formula \eqref{eqGF2}, taking into account that $l=s$ on $\partial L_s$, we write
\begin{eqnarray}\label{eq:estIII}
 \displaystyle  \frac{1}{4} s^{-N} \int_{L_s} l^{-3} ( \mathbf{X}, \bm{\nu})_{\L^{N+1}}^2  v^{\gamma} \ dA
 &=& \displaystyle -\frac{1}{4(N-1)}D_s\left[ s^{-N+1} \int_{L_s} l^{-3} ( \mathbf{X}, \bm{\nu})_{\L^{N+1}}^2  v^{\gamma} \ dA\right]\nonumber \\ & & + \frac{1}{4(N-1)}s^{-N+1} D_s\left[ \int_{L_s} l^{-3} ( \mathbf{X}, \bm{\nu})_{\L^{N+1}}^2  v^{\gamma} \ dA\right]\nonumber\\
  &=& \displaystyle -\frac{1}{4(N-1)}D_s\left[ s^{-N+1} \int_{L_s} l^{-3} ( \mathbf{X}, \bm{\nu})_{\L^{N+1}}^2  v^{\gamma} \ dA\right]\nonumber \\ & & + \frac{1}{4(N-1)} \int_{\partial L_s} l^{-N-2} ( \mathbf{X}, \bm{\nu})_{\L^{N+1}}^2  v^{\gamma}\|\bm{\delta} l\|_{\LN}^{-1} \ d\mu\nonumber\\
   &=& \displaystyle -\frac{1}{4(N-1)}D_s\left[ s^{-N+1} \int_{L_s} l^{-3} ( \mathbf{X}, \bm{\nu})_{\L^{N+1}}^2  v^{\gamma} \ dA\right]\nonumber\\ & &  + \frac{1}{4(N-1)}D_s\left[ \int_{ L_s} l^{-N-2} ( \mathbf{X}, \bm{\nu})_{\L^{N+1}}^2  v^{\gamma} \ dA \right].
\end{eqnarray}
Treating $(I_6)$ similarly yields
\begin{eqnarray}\label{eq:estIV}
 \displaystyle \frac{1}{2} s^{-N-1} \int_{L_s}  l^{-2}(\mathbf{X},\bm{\nu})_{\LN}^2 v^{\gamma}\ dA
& = & \displaystyle -\frac{1}{2N}D_s\left[ s^{-N} \int_{L_s} l^{-2} ( \mathbf{X}, \bm{\nu})_{\L^{N+1}}^2  v^{\gamma} \ dA\right] \nonumber\\
& &+ \frac{1}{2N}D_s\left[ \int_{ L_s} l^{-N-2} ( \mathbf{X}, \bm{\nu})_{\L^{N+1}}^2  v^{\gamma} \ dA \right].
\end{eqnarray}
Therefore, from \eqref{eq2monform}--\eqref{eq:estIV}, we readily obtain the estimate
\begin{multline}\label{eq3monform}
\displaystyle D_s\left\{s^{-N} \int_{L_s} v^\gamma \ dA\right\} 
\leq \displaystyle -C \int_{L_s} \frac{1}{2} s^{-N-1} (s^2-l^2)\left(\sum_{i,j=1}^N u_{ij}^2 + \sum_{j=1}^N\left(\sum_{i=1}^N \nu_i u_{ij}\right)^2\right) v^{\gamma-2} \ dA\\
+  \displaystyle \frac{3}{2}|\rho|_{q,K_s}^2 s^{1-\beta}\left(s^{-N}\int_{L_s} v^{\gamma} \ dA\right)^{\frac{q-2}{q}} + \frac{1}{2} |\rho|_{q,K_s} s^{-\frac{\beta}{2}}\left(s^{-N}\int_{L_s} v^{\gamma} \ dA\right)^{\frac{q-1}{q}}\\
- \displaystyle  D_s\left\{\int_{L_s} \left[ \frac{s^{-N}}{2N} l^{-2} + \frac{s^{-N+1}}{4(N-1)} l^{-3}+\left(1-\frac{3N-2}{4N(N-1)}\right)l^{-N-2}\right]( \mathbf{X}, \bm{\nu})_{\L^{N+1}}^2 v^\gamma \ dA\right\}.
\end{multline}
Setting the notations
$$\psi(s):=s^{-N} \int_{L_s} v^\gamma \ dA$$
and
$$F(s):= \int_{L_s} \left[ \frac{s^{-N}}{2N} l^{-2}+ \frac{s^{-N+1}}{4(N-1)} l^{-3}+\left(1-\frac{3N-2}{4N(N-1)}\right)l^{-N-2}\right]( \mathbf{X}, \bm{\nu})_{\L^{N+1}}^2 v^\gamma \ dA, $$
we can rewrite \eqref{eq3monform} as
\begin{multline}\label{eq5monform}
\displaystyle  \psi^\prime(s)  -\frac{3}{2}|\rho|_{q,K_s}^2 s^{1-\beta}\left(\psi(s)\right)^{\frac{q-2}{q}} - \frac{1}{2} |\rho|_{q,K_s} s^{-\frac{\beta}{2}}\left(\psi(s)\right)^{\frac{q-1}{q}}\\
\leq \displaystyle -C \int_{L_s} \frac{1}{2} s^{-N-1} (s^2-l^2)\left(\sum_{i,j=1}^N u_{ij}^2 + \sum_{j=1}^N\left(\sum_{i=1}^N \nu_i u_{ij}\right)^2\right) v^{\gamma-2} \ dA - F^\prime(s).
\end{multline}
Now recall that 
for every $h \in L^1(M)$, we have 
\beq\label{eq7monformstep}
 \int_{L_s} h \ dA= \int_M U(s-l) h \ dA,
\eeq
where $U:\R\to[0,1]$ is the Heaviside (Unit Step) function. 
Moreover, since $v^\gamma$ is continuous at $x_0=0$ and $l$ approximates the geodesic distance in $M$ for $|x|$ small, we have
\beq\label{eq7monformGeo}
\psi(s) \to \omega_N v^{\gamma}(0),\ \ \hbox{as}\ s \to 0^+,
\eeq
see \cite[Section 2]{BS}.
In addition, since $M$ is $C^2$ and strictly spacelike, we have $(\mathbf{X},\bm{\nu})_{\L^{N+1}}=O(|x|^2)$ as $|x|\to 0$. We hence infer that 
\begin{eqnarray}
\label{eq8monformGeo}
 \lim_{s\to 0^+}\int_{L_s} ( \mathbf{X},\bm{\nu})_{\L^{N+1}}^2 l^{-N-2}\ dA =0, \nonumber\\ 
 \lim_{s\to 0^+}\int_{L_s} ( \mathbf{X},\bm{\nu})_{\L^{N+1}}^2 s^{-N}l^{-2}\ dA = 0, \nonumber\\  
 \lim_{s\to 0^+}\int_{L_s} ( \mathbf{X},\bm{\nu})_{\L^{N+1}}^2 s^{-N+1}l^{-3}\ dA =0.
 \end{eqnarray}
We now integrate \eqref{eq5monform} from $0$ to $t\in (0,R]$. Exploiting \eqref{eq7monformstep}--\eqref{eq8monformGeo}, Fubini's theorem and observing that $|\rho|_{q,K_s}\leq |\rho|_{q,K_R}$, for any $s\in(0,R]$, we obtain
\begin{multline}\label{eq8monform}
\displaystyle\psi(t)- \omega_Nv^\gamma(0) -  \frac{3}{2}|\rho|_{q,K_R}^2 \int_0^t  s^{1-\beta} \psi^{\frac{q-2}{q}}(s) \ ds -  \frac{1}{2}|\rho|_{q,K_R} \int_0^t  s^{-\frac{\beta}{2}} \psi^{\frac{q-1}{q}}(s) \ ds \\
 \leq\displaystyle -C  \int_{L_t} S_t(l) \left(\sum_{i,j=1}^N u_{ij}^2 + \sum_{j=1}^N\left(\sum_{i=1}^N \nu_i u_{ij}\right)^2  \right)v^{\gamma-2}\,  dA - F(t),
\end{multline}
where
$$S_t(l):=\int_0^t  \frac{1}{2} s^{-N-1} (s^2-l^2) U(s-l) ds.$$
A direct computation shows that
\begin{equation}\label{eq:explcompSRl}
S_t(l)=\frac{1}{N(N-2)} l^{2-N}+\frac{1}{2N}l^2t^{-N}-\frac{1}{2(N-2)}t^{2-N}>0
\end{equation}
for $0<l<t$. In addition, since $N\geq 3$, we see that $F(t)\geq 0$ for all $t\in(0,R]$. From these facts we deduce that the right-hand side of \eqref{eq8monform} is non-positive, and taking into account that $v\leq1$, we get the estimate
\begin{equation}\label{eq12monform}
\psi(t)\leq \omega_N +   \frac{3}{2}|\rho|_{q,K_R}^2 \int_0^t  s^{1-\beta} \psi^{\frac{q-2}{q}}(s) \ ds +  \frac{1}{2}|\rho|_{q,K_R} \int_0^t  s^{-\frac{\beta}{2}} \psi^{\frac{q-1}{q}}(s) \ ds \ \ \forall t \in (0,R].
\end{equation}
Moreover, in view of \eqref{eq7monformGeo}, $\psi$ has a continuous extension on $[0,R]$ by setting $\psi(0):=\omega_N v^\gamma(0)$ and thus the inequality
\eqref{eq12monform} holds for all $t\in[0,R]$. Therefore, applying Lemma \ref{lemtech:hyp} with $T=R$, $C_0=\omega_N$, $C_1= \frac{3}{2}|\rho|_{q,K_R}^2$, $C_2= \frac{1}{2}|\rho|_{q,K_R} $, we conclude that
\begin{equation}\label{eq12amonform}
\psi(t) \leq  \left(\omega_N^{\frac{1}{q}} + \frac{3\ \omega_N^{-\frac{1}{q}} |\rho|_{q,K_R}^2}{2q(2-\beta)}\ t^{2-\beta} +\frac{|\rho|_{q,K_R}}{2q(1-\frac{\beta}{2})}\ t^{1-\frac{\beta}{2}} \right)^q\ \ \forall t\in [0,R].
\end{equation}
Now, observe that it easily follows from \eqref{eq:explcompSRl} that  for $0<l<R/2$,
 $$S_R(l)>c(N)R^{2-N},$$ 
 where $c(N)>0$ depends on $N$ only. Then, from \eqref{eq8monform} and \eqref{eq12amonform}, recalling that $F(R)\geq 0$, $\nu_i=\frac{u_i}{v}$ and $dA=v dx$, we get
\begin{eqnarray}\label{eq14monform}
\displaystyle  \omega_Nv^\gamma(0)&\geq & \displaystyle R^{-N} \int_{K_{R}} v^{\gamma+1} \ dx + c(N)R^{2-N}\int_{K_{R/2}}v^{\gamma-1}\left[\sum_{i,j=1}^N u_{ij}^2 + v^{-2}\sum_{j=1}^N\left(\sum_{i=1}^N u_{ij}u_i\right)^2\right] \ dx\nonumber\\
 &&\displaystyle - \frac{3}{2}|\rho|_{q,K_R}^2 \int_0^R  s^{1-\beta}  \left(\omega_N^{\frac{1}{q}} + \frac{3\ \omega_N^{-\frac{1}{q}} |\rho|_{q,K_R}^2}{2q(2-\beta)}\ s^{2-\beta} +\frac{|\rho|_{q,K_R}}{2q(1-\frac{\beta}{2})}\ s^{1-\frac{\beta}{2}} \right)^{q-2} ds\nonumber \\
 &&\displaystyle - \frac{1}{2}|\rho|_{q,K_R} \int_0^R  s^{-\frac{\beta}{2}}  \left(\omega_N^{\frac{1}{q}} + \frac{3\ \omega_N^{-\frac{1}{q}} |\rho|_{q,K_R}^2}{2q(2-\beta)}\ s^{2-\beta} +\frac{|\rho|_{q,K_R}}{2q(1-\frac{\beta}{2})}\ s^{1-\frac{\beta}{2}} \right)^{q-1} ds.
\end{eqnarray}
Up to the translation argument pointed out at the beginning of the proof and noticing that
\beq\label{eq:finalmainteo}
v^{-2}\sum_{j=1}^N\left(\sum_{i=1}^N u_{ij}u_i\right)^2=\frac{v^{-2}}{4}\sum_{j=1}^N\left(\frac{\partial}{\partial x_j}|\nabla u|^2\right)^2=\frac{1}{4} \sum_{j=1}^N \left[\frac{\frac{\partial}{\partial x_j}\left(-|\nabla u|^2\right)}{\sqrt{1-|\nabla u|^2}} \right]^2=|\nabla v|^2,
\eeq
we finaly deduce \eqref{eqTesiProp2} from \eqref{eq14monform}.
\end{proof}

\section{Haarala's gradient estimate}

In this section, we first give a detailed version of Haarala's proof of \cite[Theorem 3.5]{Haa}, which is reminiscent of  \cite[Theorem 3.5]{BS}, and whose key final ingredient is a Moser iteration technique. Then, under further assumptions, we provide an estimate in Proposition \ref{prop:Lqestnu} on the $L^q$ norm of the function $\nu$ defined by
\begin{equation}\label{eq:nu}
\nu(x):= \frac{1}{\sqrt{1-|\nabla u(x)|^2}}\ \ \ x \in \RN,
\end{equation}
where $u \in C^3(\RN)$ is a given strictly spacelike function. As seen in Section 2, the quantity $\nu$ has a geometrical interpretation as $\nu=\nu_{N+1}=v^{-1}$ is the $(N+1)$-th component in $\L^{N+1}$ of the Gauss map $\bm{\nu}$ associated to $M$, see \eqref{eq:Gausscomponents}, and it plays a crucial role to deduce  gradient estimates for \eqref{eq:BI2} (see \cite{BS, BIA2,GH}). In terms of $\nu$ we have the following.

\begin{theorem}[Haarala \cite{Haa}]\label{teo:maingradestimate2}
Assume $N\geq 3$, $q>N$ and $\rho \in C^1(\RN)$. Let $u\in C^3(\RN)$ be a strictly spacelike classical solution to \eqref{eq:BI2} and let $\nu$ be the function defined by \eqref{eq:nu}. 
There exists  a positive constant $c$ depending only on $N$ and $q$ such that  for any $x_0\in \RN$, $R>0$ it holds,
$$\sup_{B_{R/2}(x_0)} \nu \leq c  \left[\left( \Mint_{B_R(x_0)} \nu^q  dx\right)^{\frac{N}{q(q-N)}} + R^{\frac{N}{q-N}} \left( \Mint_{B_R(x_0)} |\rho|^q  dx\right)^{\frac{N}{q(q-N)}} \right] \left( \Mint_{B_R(x_0)} \nu^q  dx\right)^{\frac{1}{q}}.$$
\end{theorem}

\begin{proof}[Proof of Theorem \ref{teo:maingradestimate2}]
Let $F\in C^0\left(\overline{B_1(0)}\right)\cap C^\infty(B_1(0))$ be defined by $F(y):=1-\sqrt{1-|y|^2}$, $y\in \overline{B_1(0)}$. By direct computation we easily check that
\beq\label{eq:FgradHess}
  \nabla F(y)=\frac{y}{\sqrt{1-|y|^2}}, \ \  D^2F(y)= \frac{1}{{\sqrt{1-|y|^2}}} \mathbb{I}_N + \frac{1}{{\sqrt{(1-|y|^2)^3}}} y\otimes y,
  \eeq
for any $y=(y_1,\ldots, y_N)\in B_1(0)$, where $\mathbb{I}_N$ is the identity matrix in $\RN$, $y\otimes y$ is the symmetric matrix whose $(i,j)$ entry is given by $y_iy_j$, for $i,j=1,\ldots,N$.
In particular since $u$ is strictly spacelike, and recalling that $\nu$ is given by \eqref{eq:nu}, we have
\beq\label{eq0:unifboundD2F}
D^2F(\nabla u)= \nu \mathbb{I}_N+ \nu^3 \nabla u\otimes\nabla u.
\eeq
One then easily checks that $D^2F(\nabla u)$ is positive definite for any $x\in \RN$ as
\beq\label{eq:unifboundD2F}
\nu |\xi|^2\leq  (D^2F(\nabla u) \xi, \xi)_{\RN} \leq \nu^3 |\xi|^2 \ \ \ \forall \xi \in \RN.
\eeq
Now, since $u$ is of class $C^3$, strictly spacelike and satisfies classically \eqref{eq:BI2}, we can differentiate  \eqref{eq:BI2} with respect to the variable $x_i$. Taking into account \eqref{eq:FgradHess}, we easily check that
\begin{equation}\label{eq:BIdiff}
-\operatorname{div}\left(D^2F (\nabla u) \nabla u_i\right)=\rho_i \ \ \hbox{in $\RN$},
\end{equation}
in the classical sense, for any $i=1,\ldots,N$. In particular, we have
 \beq\label{eq:weakformD2F}
\int_{\RN} (D^2F(\nabla u) \nabla u_i, \nabla \psi)_{\RN} \ dx=  \int_{\RN} \rho_i \psi \ dx \ \ \ \forall \psi \in C^1_c(\RN).
\eeq
Taking $\psi=\varphi u_i$ as test function in \eqref{eq:weakformD2F}, where $\varphi \in C^1_c(\RN)$ is such that $\varphi\geq 0$, and summing over all indices $i=1,\ldots, N$, we infer that
\beq\label{eq2:th35}
\sum_{i=1}^N \int_{\RN} (D^2F(\nabla u) \nabla u_i, \nabla \varphi)_{\RN} u_i\ dx =  \int_{\RN}  \varphi (\nabla u,\nabla \rho)_{\RN} \ dx - \sum_{i=1}^N \int_{\RN} (D^2F(\nabla u) \nabla u_i, \nabla u_i)_{\RN} \varphi\ dx.
\eeq
Now, thanks to \eqref{eq:unifboundD2F} and since $\varphi\geq 0$, we deduce that 
$$(D^2F(\nabla u) \nabla u_i, \nabla u_i)_{\RN} \varphi \geq 0,$$ 
for any $i=1,\ldots,N$, and therefore \eqref{eq2:th35} leads to the inequality
\beq\label{eq3:th35}
\sum_{i=1}^N \int_{\RN} (D^2F(\nabla u) \nabla u_i, \nabla \varphi)_{\RN} u_i\ dx \leq  \int_{\RN}  \varphi (\nabla u,\nabla \rho)_{\RN} \ dx.
\eeq
Observe that 
$$\varphi (\nabla u,\nabla \rho)_{\RN}=\nu (\nabla u,\nabla (\rho\varphi \nu^{-1}))_{\RN}-(\nabla u,\nabla \varphi)_{\RN}\rho + \nu^{-1} (\nabla u,\nabla \nu)_{\RN}\varphi\rho.$$
Integrating by parts and taking into account \eqref{eq:BI2}, we obtain
\beq\label{eq3bis:th35}
\int_{\RN}  \varphi (\nabla u,\nabla \rho)_{\RN} \ dx= \int_{\RN}  \nu^{-1} \varphi \rho^2 \ dx  -  \int_{\RN}   (\nabla u,\nabla \varphi)_{\RN} \rho\ dx +  \int_{\RN}  \nu^{-1} (\nabla u,\nabla \nu)_{\RN} \varphi \rho\ dx.
\eeq
Let $A:=\nu^{-3} D^2F(\nabla u)$. Since $\nabla \nu=\nu^3 \sum_{i=1}^N u_i \nabla u_i$, it follows from \eqref{eq3:th35} and \eqref{eq3bis:th35} that
\beq\label{eq4:th35}
\int_{\RN} (A \nabla \nu, \nabla \varphi)_{\RN} \ dx \leq  \int_{\RN}  \nu^{-1} \varphi \rho^2 \ dx  -  \int_{\RN}   (\nabla u,\nabla \varphi)_{\RN} \rho\ dx +  \int_{\RN}  \nu^{-1} (\nabla u,\nabla \nu)_{\RN} \varphi \rho\ dx,
\eeq
 for any $\varphi \in C^1_c(\RN)$ such that $\varphi\geq 0$. Let $p>N$, $\eta\in C^1_c(\RN)$ and plug $\varphi=\eta^2 \nu^{p-1}$ in \eqref{eq4:th35}. As 
 $$\nabla \varphi= 2\eta \nu^{p-1} \nabla \eta + (p-1) \eta^2 \nu^{p-2}\nabla \nu,$$ 
this yields 
\begin{eqnarray}
 \label{eq5:th35}
\displaystyle (p-1)\!\! \int_{\RN} \!\!\eta^2 \nu^{p-2}(A \nabla \nu, \nabla \nu)_{\RN} \ dx \!\!\!&\leq&\!\!\!\!\!\displaystyle  \int_{\RN} \eta^2  \nu^{p-2} \rho^2 \ dx - (p-2) \int_{\RN} \eta^2 \nu^{p-2}  (\nabla u,\nabla \nu)_{\RN} \rho\ dx \nonumber\\
  &&\!\!\!\!\! - 2 \int_{\RN} \eta \nu^{p-1}  (\nabla u,\nabla \eta)_{\RN} \rho\ dx - 2 \int_{\RN} \eta \nu^{p-1}  (A\nabla \nu,\nabla \eta)_{\RN}\ dx. 
\end{eqnarray}
By definition, we have $A(x)= (1-|\nabla u(x)|^2)\mathbb{I}_N+ \nabla u(x)\otimes\nabla u(x)$, $x\in\RN$, 
and therefore $A\nabla u= \nabla u$. 
Moreover, in view of \eqref{eq:unifboundD2F}, for any $x\in\RN$, it holds
\beq\label{eq6bis:th35}
\nu(x)^{-2} |\xi|^2\leq  (A(x) \xi, \xi)_{\RN} \leq  |\xi|^2 \ \ \ \forall \xi \in \RN.
\eeq
In particular, for any $x\in \RN$ the map $(\xi,\zeta)\mapsto (A(x)\xi,\zeta)_{\RN}$ is a inner product on $\RN$, and by the Cauchy-Schwarz inequality we have
\beq\label{eq7:th35}
|(A \xi, \zeta)_{\RN}| \leq \sqrt{(A \xi, \xi)_{\RN}}  \sqrt{(A \zeta, \zeta)_{\RN}} \ \ \ \forall \xi,\zeta \in \RN.
\eeq
Replacing $\nabla u$ by $A\nabla u$, we infer from \eqref{eq7:th35} that
$$|(\nabla u, \nabla \nu)_{\RN}|=|(A\nabla u, \nabla \nu)_{\RN}|\leq \sqrt{(A\nabla u, \nabla u)_{\RN}} \sqrt{(A\nabla \nu, \nabla \nu)_{\RN}}= |\nabla u|  \sqrt{(A\nabla \nu, \nabla \nu)_{\RN}}.$$
Using the same trick to estimate $|(\nabla u, \nabla \eta)_{\RN}|$, exploiting the fact that $|\nabla u|\leq1$ and using \eqref{eq7:th35} to bound $|(A\nabla \nu,\nabla \eta)_{\RN}|$, we deduce from \eqref{eq5:th35} that
\begin{eqnarray}
\label{eq8:th35}
\displaystyle (p-1) \int_{\RN} \eta^2 \nu^{p-2}(A \nabla \nu, \nabla \nu)_{\RN} \ dx &\leq&\displaystyle  \int_{\RN} \eta^2  \nu^{p-2} \rho^2 \ dx \nonumber\\
  &&\displaystyle + (p-2) \int_{\RN} \eta^2 \nu^{p-2}  |\rho|  \sqrt{(A\nabla \nu, \nabla \nu)_{\RN}}  \ dx \nonumber\\
    &&\displaystyle + 2 \int_{\RN} |\eta| \nu^{p-1}|\rho|  \sqrt{(A\nabla \eta, \nabla \eta)_{\RN}}  dx\nonumber\\
     &&\displaystyle + 2 \int_{\RN} |\eta| \nu^{p-1}   \sqrt{(A\nabla \nu, \nabla \nu)_{\RN}}  \sqrt{(A\nabla \eta, \nabla \eta)_{\RN}}\ dx. 
\end{eqnarray}
We now use Young inequality \eqref{eq:elementinterpolineq} to estimate the last three terms in \eqref{eq8:th35}.
For the reader's convenience, and since it will be crucial to keep track of the dependence on the parameter $p$ in the constants, we give the full details of the estimates. First, applying  \eqref{eq:elementinterpolineq} with $\e=1/(p-1)$ we get 
$$
\displaystyle (p-2) \eta^2\nu^{p-2}  |\rho|  \sqrt{(A\nabla \nu, \nabla \nu)_{\RN}}
\leq \displaystyle \frac{(p-2)^2}{(p-1)}\eta^2 \nu^{p-2}\rho^2 +\frac{p-1}{4}\eta^2  \nu^{p-2} (A\nabla \nu, \nabla \nu)_{\RN}, $$
while 
\begin{eqnarray*}
2\displaystyle  |\eta| \nu^{p-1}  |\rho|  \sqrt{(A\nabla \eta, \nabla \eta)_{\RN}}&=&\displaystyle   \left(\sqrt{p-1}|\eta| \nu^{\frac{p}{2}-1}|\rho|\right) \left(\frac{2}{\sqrt{p-1}} \nu^{\frac{p}{2}} \sqrt{(A\nabla \eta, \nabla \eta)_{\RN}}\right)\\
&\leq& \displaystyle (p-1)\eta^2 \nu^{p-2}\rho^2 +\frac{1}{p-1}  \nu^{p} (A\nabla \eta, \nabla \eta)_{\RN}.
\end{eqnarray*}
Similarly, we infer that
$$ 2 \nu^{\frac{p}2} \sqrt{(A\nabla \eta, \nabla \eta)_{\RN}}  |\eta| \nu^{\frac{p}2-1} \sqrt{(A\nabla \nu, \nabla \nu)_{\RN}}  \leq   \frac{4}{p-1} \nu^{p} (A\nabla \eta, \nabla \eta)_{\RN}  +\frac{p-1}{4} \eta^2 \nu^{p-2} (A\nabla \nu, \nabla \nu)_{\RN}.$$
Summing up, we deduce from \eqref{eq8:th35} that
\begin{eqnarray}
\label{eq9:th35}
\displaystyle \frac{(p-1)}{2} \int_{\RN} \eta^2 \nu^{p-2}(A \nabla \nu, \nabla \nu)_{\RN} \ dx &\leq&\displaystyle  \frac{5}{p-1}\int_{\RN}  \nu^{p} (A\nabla \eta, \nabla \eta)_{\RN} \ dx \nonumber\\
  &&\displaystyle + \left[\frac{(p-2)^2+(p-1)^2}{p-1}+1\right] \int_{\RN} \eta^2 \nu^{p-2}  \rho^2 \ dx.
  \end{eqnarray}

From \eqref{eq9:th35} and since $p>N$, $N\geq 3$, then by elementary algebraic considerations we get
 \beq\label{eq10:th35}
\displaystyle \int_{\RN} \eta^2 \nu^{p-2}(A \nabla \nu, \nabla \nu)_{\RN} \ dx \leq\displaystyle  \frac{10}{(p-1)^2}\int_{\RN}  \nu^{p} (A\nabla \eta, \nabla \eta)_{\RN} \ dx  + 5 \int_{\RN} \eta^2 \nu^{p-2}  \rho^2 \ dx.
\eeq
Now, applying Sobolev's inequality to $\phi:=\eta \nu^{\frac{p-2}{2}} \in C^1_c(\RN)$, taking into account that 
$$
\displaystyle\left|\nabla (\eta \nu^{\frac{p-2}{2}})\right|^2 
\leq\displaystyle \frac{1}{2} (p-2)^2 \eta^2 \nu^{p-4} |\nabla \nu|^2 + 2\nu^{p-2} |\nabla \eta|^2,
$$
and exploiting \eqref{eq6bis:th35} we deduce that
\beq\label{eq11:th35}
\left|\eta \nu^{\frac{p-2}{2}}\right|_{2^*}^2 
\leq  \frac{c_0^2(p-2)^2}{2}  \int_{\RN} \eta^2 \nu^{p-2} (A\nabla \nu, \nabla \nu)_{\RN} \ dx +  2c_0^2  \int_{\RN} \nu^{p-2} |\nabla \eta|^2 \ dx,
\eeq
where $c_0$ is the Sobolev constant for the embedding of $\D\hookrightarrow L^{2^*}(\RN)$, in particular $c_0$ depends only on $N$. Finally, combining \eqref{eq11:th35} with \eqref{eq10:th35}, exploiting again \eqref{eq6bis:th35} and taking into account that $\nu\geq 1$, we infer that
\beq\label{eq12:th35}
\left|\eta \nu^{\frac{p-2}{2}}\right|_{2^*}^2 
\leq 7 c_0^2  \int_{\RN} \nu^{p} |\nabla\eta|^2 \ dx +  3c_0^2 (p-2)^2  \int_{\RN} \eta^2\nu^{p-2} \rho^2 \ dx.
\eeq
Take $k \in \N$. Set $R_k:=(\frac{1}{2}+\frac{1}{2^{k+1}})R$, and $B_k:=B_{R_k}(x_0)$. Clearly, by definition, we have $B_0=B_R(x_0)$, $B_{k+1}\subset B_k$ and we easily check that $\frac{R_k}{R_{k+1}}=\frac{2^{k+1}+2}{2^{k+1}+1}$, $\frac{m_N(B_k)}{m_N(B_{k+1})}\leq C$, for some positive constant depending only on $N$, where $m_N(\cdot)$ denotes the Lebesgue measure in $\RN$. Let $\eta_k \in C^1_c(\RN)$ be such that $0\leq \eta_k\leq 1$, $\eta_k\equiv 1$ in $B_{k+1}$, $\supp(\eta_k)\subset B_k$ ($\supp(\eta_k)$ is the support of $\eta_k$) 
and 
\beq\label{eq:propcutoffetam}
|\nabla \eta_k|\leq \frac{2}{R_k-R_{k+1}}=2^{k+3}R^{-1}.
\eeq
Taking $\eta=\eta_k$ in \eqref{eq12:th35}, the estimate \eqref{eq:propcutoffetam} leads to
\beq\label{eq13:th35}
\left(\Mint_{B_{k+1}} \nu^{(p-2)\frac{N}{N-2}} \ dx\right)^{\frac{N-2}{N}} \leq 4^{k+3}c_1^2  \Mint_{B_k} \nu^{p}  \ dx +  c_1^2 (p-2)^2 R^2 \Mint_{B_k} \nu^{p-2} \rho^2 \ dx.
\eeq
where $c_1$ is a positive constant depending only on $N$. 

Now, fix once for all a number $q>N$ and let $p\geq q$. Using H\"older's inequality to estimate the integrals in the right-hand side of \eqref{eq13:th35} we get that
\begin{equation}\label{eq14:th35}
\begin{array}{lll}
 \displaystyle \Mint_{B_k} \nu^{p} \ dx &\leq& \displaystyle \left( \Mint_{B_k} \nu^{q} \ dx \right)^{\frac{2}{q}} \left( \Mint_{B_k} \nu^{(p-2)\frac{q}{q-2}} \ dx \right)^{\frac{q-2}{q}}\\[12pt]
 \displaystyle \Mint_{B_k} \nu^{p-2}\rho^2 \ dx &\leq&\displaystyle  \left( \Mint_{B_k} |\rho|^{q} \ dx \right)^{\frac{2}{q}} \left( \Mint_{B_k} \nu^{(p-2)\frac{q}{q-2}} \ dx \right)^{\frac{q-2}{q}}.
  \end{array}
  \end{equation} 
 Then from \eqref{eq13:th35}, \eqref{eq14:th35}, and observing that $\frac{|B_R(x_0)|}{|B_k|}\leq \frac{|B_R(x_0)|}{|B_{R/2}(x_0)|}\leq 2^N$ for all $k\in\N$ (which is used to bound the first integral means in the right-hand sides of \eqref{eq14:th35} with the same integral means on the whole $B_R(x_0)$), then, after elementary computations (taking into account that $4^k\geq 1$, $(p-2)^2\geq 1$), we deduce that
 \beq\label{eq15:th35}
\left(\Mint_{B_{k+1}} \nu^{(p-2)\frac{N}{N-2}} \ dx\right)^{\frac{N-2}{N(p-2)}} \leq 4^{\frac{k}{p-2}}g^{\frac{1}{p-2}}(p-2)^{\frac{2}{p-2}} \left( \Mint_{B_k} \nu^{(p-2)\frac{q}{q-2}} \ dx \right)^{\frac{(q-2)}{(p-2)q}},
\eeq
where we have set $$g:= 4^3 2^{\frac{2N}{q}}c_1^2 \left( \Mint_{B_R(x_0)} \nu^{q} \ dx \right)^{\frac{2}{q}}+c_1^22^{\frac{2N}{q}} R^2\left( \Mint_{B_R(x_0)} |\rho|^{q} \ dx \right)^{\frac{2}{q}}.$$
In particular notice that $g$ is independent of $p$ and $k$.
Now, setting $\alpha:=\frac{N (q-2)}{(N-2)q}$ we define
$$p_k:=\alpha^k (q-2)+2, \ \ \Phi_k:=\left( \Mint_{B_k} \nu^{\alpha^k q} \ dx \right)^{\frac{1}{\alpha^k q}}, \ \ k\in \N.$$ Clearly $p_0=q$, in addition, since $q>N$, we check that $\alpha>1$ and thus $p_k\geq q$ for any $k \in \N$, $p_k\to +\infty$, as $k\to+\infty$. Taking $p=p_k$ in \eqref{eq15:th35} then by construction, and after easy algebraic manipulations, we infer that for any $k\in \N$ 
$$\Phi_{k+1} \leq 2^{\frac{2}{q-2} k \alpha^{-k}} \alpha^{\frac{2}{q-2}k\alpha^{-k}} \left[(q-2) \sqrt{g}\right]^{\frac{2}{q-2} \alpha^{-k}} \Phi_k.$$
Then, arguing by induction and passing to the limit as $k\to+\infty$ we deduce that
 \beq\label{eq16:th35}
\lim_{k\to +\infty} \Phi_{k} \leq {(2\alpha)}^{\frac{2}{q-2} \sum_{j=1}^\infty j \alpha^{-j}} \left[(q-2) \sqrt{g}\right]^{\frac{2}{q-2}  \sum_{j=0}^\infty \alpha^{-j}} \Phi_0. 
\eeq
Now, since $\alpha=\frac{N (q-2)}{(N-2)q}>1$ the two series in \eqref{eq16:th35} converge, and by direct computation we have
 \beq\label{eq17:th35}
\frac{2}{q-2} \sum_{j=0}^\infty \alpha^{-j}= \frac{N}{q-N}, \ \ \frac{2}{q-2} \sum_{j=1}^\infty j \alpha^{-j} = \frac{qN(N-2)}{2(q-N)^2}.
\eeq
Finally, as $\lim_{k\to +\infty} \Phi_{k}=|\nu|_{\infty, B_{R/2}(x_0)}$, $\Phi_0= \left( \Mint_{B_R(x_0)} \nu^q  dx\right)^{\frac{1}{q}}$, from \eqref{eq16:th35}, \eqref{eq17:th35}, recalling the definition of $g$ and after elementary computations we deduce that
$$\sup_{B_{R/2}(x_0)} \nu \leq c_2  \left[\left( \Mint_{B_R(x_0)} \nu^q  dx\right)^{\frac{N}{q(q-N)}} + R^{\frac{N}{q-N}} \left( \Mint_{B_R(x_0)} |\rho|^q  dx\right)^{\frac{N}{q(q-N)}} \right] \left( \Mint_{B_R(x_0)} \nu^q  dx\right)^{\frac{1}{q}},$$
for some positive constant $c_2$ depending only on $N$ and $q$. The proof is complete.
\end{proof}
When $\rho$ is in $L^q(\RN)$ with $q>N$, if the solution $u$ is globally bounded and there exists $\nu_0>1$ such that $(\nu-\nu_0)_+$ has compact support, then an $L^q$-estimate for $(\nu-\nu_0)_+$ holds. This is the content of the next result which is \cite[Theorem 3.6]{Haa} revisited.

\begin{proposition}\label{prop:Lqestnu}
Assume $N\geq 3$, $q>N$, $\rho \in L^q(\RN)\cap C^1(\RN)$, and let $u\in C^3(\RN)\cap L^\infty(\RN)$ be a strictly spacelike classical solution to \eqref{eq:BI2}. Assume that there exists a number $\nu_0>1$ such that $(\nu-\nu_0)_+$ has compact support. There exists a positive constant $c$ depending only on $q$ and $\nu_0$ such that
\beq\label{eq:tesiTh36}
|(\nu - \nu_0)_+|_q\leq c |u|_\infty |\rho|_q.
\eeq
\end{proposition}
\begin{proof}
Since $u$ is a classical solution to \eqref{eq:BI2}, a standard integration by parts gives
\beq\label{eq1:th36}
\int_{\RN} \nu (\nabla u, \nabla \varphi)_{\RN} \ dx = \int_{\RN}  \varphi  \rho \ dx \ \ \forall \varphi \in C^1_c(\RN).
\eeq
In particular, thanks to the assumptions and by density we infer that \eqref{eq1:th36} holds true also for any $\varphi\in W^{1,p}_0(\Omega)$, where $\Omega$ is a given smooth bounded domain of $\RN$ and $p\geq 1$. Hence, since we are assuming that $(\nu-\nu_0)_+$ has compact support, by taking $\varphi=u\nu^{-1}(\nu-\nu_0)_+^{q}$ in \eqref{eq1:th36} we get 
\begin{eqnarray}\label{eq2:th36}
\displaystyle \int_{\RN}(\nu-\nu_0)_+^{q} |\nabla u|^2 \ dx \!\!\!\!&=&\!\!\!\! \displaystyle \int_{\RN} u \nu^{-1} (\nu-\nu_0)_+^{q} (\nabla u, \nabla \nu)_{\RN}   -q\int_{\RN} u (\nu-\nu_0)_+^{q-1} (\nabla u, \nabla \nu)_{\RN}  \ dx\nonumber \\
&&\displaystyle +\int_{\RN} u \nu^{-1} (\nu-\nu_0)_+^{q} \rho  \ dx.
\end{eqnarray}
From \eqref{eq2:th36}, recalling $A=\nu^{-3} D^2F(\nabla u)$ and the fact that $A\nabla u= \nabla u$, using the ellipticity bounds in \eqref{eq7:th35}, and observing that $\nu^{-1} (\nu-\nu_0)_+ \leq 1$, we obtain
\beq\label{eq3:th36}
\begin{array}{lll}
\displaystyle \int_{\RN}(\nu-\nu_0)_+^{q} |\nabla u|^2 \ dx &\leq & \displaystyle (q+1)\int_{\RN} |u| (\nu-\nu_0)_+^{q-1} |\nabla u| \sqrt{(A\nabla \nu, \nabla \nu)_{\RN}}  \ dx\\[12pt]
&&\displaystyle +\int_{\RN} |u| (\nu-\nu_0)_+^{q-1} |\rho|  \ dx.
\end{array}
\eeq
Now, observe that saying $\nu\geq \nu_0$ is equivalent to $|\nabla u|^2\geq \gamma_{\nu_0}:=1-\frac{1}{\nu_0^2}$. As $\nu_0>1$ we have $\gamma_{\nu_0} \in (0,1)$ so that Young inequality \eqref{eq:elementinterpolineq} implies
\begin{eqnarray}\label{eq4:th36}
\displaystyle |u| (\nu-\nu_0)_+^{q-1} |\nabla u| \sqrt{(A\nabla \nu, \nabla \nu)_{\RN}}&\leq&\displaystyle |u|^2  (\nu-\nu_0)_+^{q-2} (A\nabla \nu, \nabla \nu)_{\RN} + \frac{1}{4} (\nu-\nu_0)_+^{q} |\nabla u|^2,\nonumber\\
\displaystyle  |u| (\nu-\nu_0)_+^{q-1} |\rho| &\leq&\displaystyle \frac{1}{\gamma_{\nu_0}} |u|^2 (\nu-\nu_0)_+^{q-2}|\rho|^2 + \frac{\gamma_{\nu_0}}{4} (\nu-\nu_0)_+^{q}.
\end{eqnarray}
Then, from \eqref{eq3:th36}, \eqref{eq4:th36}, regrouping the terms, taking into account that $|\nabla u|^2\geq \gamma_{\nu_0}$ and recalling that $u\in L^\infty(\RN)$, we deduce that
\begin{eqnarray}\label{eq5:th36}
\displaystyle \int_{\RN}(\nu-\nu_0)_+^{q} \ dx &\leq & \displaystyle \frac{2q|u|^2_{\infty}}{\gamma_{\nu_0}} \int_{\RN} (\nu-\nu_0)_+^{q-2} (A\nabla \nu, \nabla \nu)_{\RN}  \ dx\nonumber\\
&&\displaystyle +  \frac{2|u|^2_{\infty}}{\gamma_{\nu_0}^2}\int_{\RN} (\nu-\nu_0)_+^{q-2} \rho^2 \ dx.
\end{eqnarray}
In order to estimate the first integral in the right-hand side of \eqref{eq5:th36}, since $\nabla \nu=\nu^3 \sum_{i=1}^N u_i \nabla u_i$, we deduce from \eqref{eq3:th35} and the definition of $A$ that for all $\varphi \in C^1_c(\RN)$ such that $\varphi\geq 0$ it holds
\beq\label{eq6:th36}
\int_{\RN} (A\nabla \nu, \nabla \varphi)_{\RN} \ dx \leq  \int_{\RN}  \varphi (\nabla u,\nabla \rho)_{\RN} \ dx.
\eeq
As pointed out at the beginning of the proof, by a density argument and since $(\nu-\nu_0)_+$ has compact support, we can take $\varphi=(\nu-\nu_0)_+^{q-1}$ as test function in \eqref{eq6:th36}. Then, arguing as in \eqref{eq9:th35} (with $\eta \in C^1_c(\RN)$, $\eta\equiv 1$ in a neighborhood of the support of $(\nu-\nu_0)_+$) we get 
\beq\label{eq7:th36}
\int_{\RN} (\nu-\nu_0)_+^{q-2}(A\nabla \nu, \nabla \nu)_{\RN} \ dx \leq 2 \int_{\RN} (\nu-\nu_0)_+^{q-2}\rho^2 \ dx.
\eeq
Plugging now \eqref{eq7:th36} into \eqref{eq5:th36} and taking into account that $\gamma_{\nu_0}\in(0,1)$, we infer that
\beq\label{eq8:th36}
\int_{\RN}(\nu-\nu_0)_+^{q} \ dx \leq   \frac{6q|u|^2_{\infty}}{\gamma_{\nu_0}^2} \int_{\RN} (\nu-\nu_0)_+^{q-2} \rho^2 \ dx.
\eeq
Finally, it follows from  H\"older's inequality  that
\beq\label{eq9:th36}
 \int_{\RN} (\nu-\nu_0)_+^{q-2} \rho^2 \ dx \leq  \left(\int_{\RN} (\nu-\nu_0)_+^{q} \ dx\right)^{\frac{q-2}{q}}  \left(\int_{\RN} |\rho|^q \ dx\right)^{\frac{2}{q}},
\eeq
and thus from \eqref{eq8:th36} and \eqref{eq9:th36}, we readily deduce the claimed inequality \eqref{eq:tesiTh36} with $c=\sqrt{\frac{6q}{\gamma_{\nu_0}^2}}$.
\end{proof}

\section{Gradient estimates for the Born-Infeld equation}

In this section we provide new gradient estimates for strictly spacelike weak solutions of the Born-Infeld equation \eqref{eq:BI} with $\rho \in L^q_{loc}(\RN)\cap L^m(\RN)$ (or  $\rho \in L^q(\RN)\cap L^m(\RN)$), $q>N$, $m\in[1,2_*]$, where we recall that $2_*$ denotes the conjugate H\"older exponent of the critical Sobolev exponent. One could have considered $\rho \in L^q_{loc}(\RN)\cap \X^*$ with $q>N$ arguing merely with little changes, but for simplicity we choose $\rho\in L^m(\RN)$ with $m\in[1,2_*]$.
We begin with a couple of preliminary technical results about the imbedding properties of the convex set $\X$ defined in \eqref{eq:defsetX}.

\begin{lemma}\label{lem:tech2}
Let $N\geq 3$ and $m\in(1,2_*]$. There exists a positive constant $c$ depending only on $N$ and $m$ such that
\begin{equation}\label{eq1:lemtech2}
|\phi|_{m^\prime} \leq c |\n \phi|_2^{2\frac{(N+1)m-N}{mN}} \ \ \forall \phi \in \X,
\end{equation}
where $m^\prime=\frac{m}{m-1}$ is the conjugate exponent of $m$.
\end{lemma}

\begin{proof}
Take $\phi \in \X$. Since $|\n \phi | \leq 1$ in $\RN$, we have $|\n \phi| \in L^k(\RN)$ for all $k\geq 2$. Therefore, by Sobolev's inequality, we infer that $\phi \in L^{k^*}(\RN)$ for $k \in [2,N[$ and
\beq\label{eq:ApplSobineq} 
|\phi|_{k^*} \leq c_0 |\n \phi|_k \leq c_0 |\n \phi|_2^{2/k},
\eeq
where $c_0=c_0(N,k)$ is the Sobolev constant for the embedding of $\mathcal{D}^{1,k}(\RN)\hookrightarrow L^{k^*}(\R^N)$.
Observe that given $m \in ]1,2_*]$, it is always possible to find $k \in [2,N[$ such that $k^*=m^\prime$, i.e.  $k=\frac{mN}{(N+1)m-N}$. Using \eqref{eq:ApplSobineq}, we then  deduce \eqref{eq1:lemtech2} with $c(N,m):=c_0\left(N, \frac{mN}{(N+1)m-N}\right)$.\\
\end{proof}

\begin{lemma}\label{lem:tech3}
Given $N\geq 3$ and $s>N$, there exist two positive constants $c_1$, $c_2$, both depending only on $N$ and $s$, such that
\begin{equation}\label{eq2:lemtech2}
|\phi|_{\infty} \leq c_1 \left(c_2 |\n \phi|_2^{2\frac{N+s}{Ns}} + |\n \phi|_2^{\frac{2}{s}} \right)\ \ \forall \phi \in \X.
\end{equation}
\end{lemma}
\begin{proof}
By Morrey-Sobolev's inequality we know that there exists a positive constant $c_1$ depending only on $N$ and $s$, such that
\beq\label{eq:MorreySobolev}
 |\phi|_\infty \leq c_1 \|\phi\|_{W^{1,s}(\RN)}, \ \forall \phi \in W^{1,s}(\RN),
\eeq
where $\|\phi\|_{W^{1,s}(\RN)}=|\phi|_s+|\n \phi|_s$ is the standard norm in $W^{1,s}(\RN)$. Now, 
on the one hand we have, as in \eqref{eq:ApplSobineq}, 
$$|\n \phi|_s\leq |\n \phi|_2^{2/s}, \ \forall\phi \in \X$$  and 
on the other hand, since $s>N$ we can always find $k \in\mathopen ]\frac{N}{2}, N\mathclose [$ such that $k^*=s$, namely $k=\frac{Ns}{N+s}$, to estimate $|\phi|_s=|\phi|_{k^*}$ by help of \eqref{eq:ApplSobineq}. Then \eqref{eq:MorreySobolev} yields
$$
|\phi|_\infty \leq c_1 (c_0 |\n \phi|_2^{2 \frac{N+s}{Ns}} + |\n \phi|_2^{2/s}), \ \forall \phi \in \X,
$$
where $c_0=c_0(N,\frac{Ns}{N+s})$ is the Sobolev constant for the embedding of $\mathcal{D}^{1,\frac{Ns}{N+s}}(\RN)\hookrightarrow L^{s}(\R^N)$, and this leads to \eqref{eq2:lemtech2} with $c_2(N,s):=c_0(N,\frac{Ns}{N+s})$.
\end{proof}

We now state and prove the first gradient estimate of this section (namely Proposition \ref{prop:mainestimate0-intro} in Section 1). We begin by studying the case $\rho \in L^m(\RN)$ with  $m\in(1, 2_*]$.

\begin{proposition}\label{prop:mainestimate0}
Assume $N\geq 3$, $q>N$ and $\rho \in L^m(\RN) \cap C^1(\RN)$ with $m\in(1,2_*]$. Let $u \in \X\cap C^3(\RN)$ be a strictly spacelike weak solution of \eqref{eq:BI}. There exist  $\gamma \in (0,\frac{1}{N})$ depending only on $N$, $c>0$  depending only on $N$ and $m$, such that for any $R>0$, $x_0\in \RN$,
  \begin{equation}\label{eq:thesisprop:mainestimateCOR0}
\displaystyle (1-|\nabla u|^2)^{\frac{\gamma}{2}}(x_0)\geq  \displaystyle \left(\frac{\omega_N}{\omega_N+ c R^{-N}|\rho|_m^{\frac{Nm}{N-m}}}\right)^{\gamma+1} - P\left(|\rho|_{q,K_R(x_0)}R^{\frac{q-N}{q}}\right),
\end{equation}
where 
\begin{eqnarray}\label{eq:defPk}
\displaystyle P(k)&:=&\!\! \left(\frac{3}{2}\right)\frac{2^{q-4} q\ \omega_N^{-\frac{2}{q}}}{q-N}k^2 + \left(\frac{3}{2}\right)^{q-1}\frac{2^{q-5} q \ \omega_N^{-2+\frac{2}{q}}}{(q-N)^{q-1}(q-1)}k^{2q-2} + \frac{2^{q-4} \omega_N^{-1}}{(q-N)^{q-1}} \left(\frac{3}{2}+\frac{1}{q-N}\right)k^{q} \nonumber\\
&&\displaystyle +\frac{2^{q-3} q\ \omega_N^{-\frac{1}{q}}}{q-N} k + \left(\frac{3}{2}\right)^{q-1}\frac{2^{q-4} q \ \omega_N^{-2+\frac{1}{q}}}{(q-N)^{q}(2q-1)}k^{2q-1}, \ \hbox{for $k\geq0$} .
\end{eqnarray}
\end{proposition}

\begin{proof}
By Lemma \ref{lem:tech3} we know that $u\in L^\infty(\RN)$ since $u\in \X$ (see also \cite[Lemma 2.1]{BDP}). Also, it is easily seen that 
$u$ is a classical solution of \eqref{eq:BI}. Then from Theorem \ref{teo:mainestimate} and Lemma \ref{lem:boundedLorentzballX}-(i), we infer that there exists a positive constant $\gamma \in (0,\frac{1}{N})$ depending only on $N$ such that for any fixed $x_0\in\RN$, $R>0$, one has
$$
 \begin{small}
\begin{array}{lll}
\displaystyle  \omega_N v^\gamma(x_0)&\geq & \displaystyle R^{-N} \int_{K_{R}(x_0)} v^{\gamma+1} \ dx\\[12pt]
 && \displaystyle-  \frac{3}{2}|\rho|_{q,K_R(x_0)}^2 \int_0^R  s^{1-\beta}  \left(\omega_N^{\frac{1}{q}} + \frac{3\ \omega_N^{-\frac{1}{q}} |\rho|_{q,K_R(x_0)}^2}{2q(2-\beta)}\ s^{2-\beta} +\frac{|\rho|_{q,K_R(x_0)}}{2q(1-\frac{\beta}{2})}\ s^{1-\frac{\beta}{2}} \right)^{q-2} ds\\[12pt]
 &&\displaystyle-  \frac{1}{2}|\rho|_{q,K_R(x_0)} \int_0^R  s^{-\frac{\beta}{2}}  \left(\omega_N^{\frac{1}{q}} + \frac{3\ \omega_N^{-\frac{1}{q}} |\rho|_{q,K_R(x_0)}^2}{2q(2-\beta)}\ s^{2-\beta} +\frac{|\rho|_{q,K_R(x_0)}}{2q(1-\frac{\beta}{2})}\ s^{1-\frac{\beta}{2}} \right)^{q-1} ds,
\end{array}
\end{small}
$$
where $v=\sqrt{1-|\nabla u|^2}$ and $\beta=\frac{2N}{q}$. Notice that we discarded the second integral in the right-hand side of \eqref{eqTesiProp2} because it is non-negative.

Using repeatedly the elementary convexity inequality 
\beq\label{eq:elementaryineq}
(a+b)^\alpha\leq 2^{\alpha-1} \left(a^\alpha + b^\alpha\right)\ \ \forall a,b>0, \forall\alpha\geq 1,
\eeq 
and taking into account that $2-\beta=2\frac{q-N}{q}$, we deduce from elementary computations that
\begin{eqnarray}\label{eq:sviluppogradestim}
\displaystyle  \omega_N v^\gamma(x_0)&\geq & \displaystyle R^{-N} \int_{K_{R}(x_0)} v^{\gamma+1} \ dx -  \left(\frac{3}{2}\right)\frac{2^{q-4} q\ \omega_N^{1-\frac{2}{q}}}{q-N}|\rho|_{q,K_R(x_0)}^2R^{2-\beta}  \nonumber \\
&&\displaystyle   - \left(\frac{3}{2}\right)^{q-1}\frac{2^{q-5} q \ \omega_N^{-1+\frac{2}{q}}}{(q-N)^{q-1}(q-1)} |\rho|_{q,K_R(x_0)}^{2(q-1)}R^{(2-\beta)(q-1)} \nonumber\\
 &&\displaystyle -  \frac{3}{2}\frac{2^{q-4}}{(q-N)^{q-1}}  |\rho|_{q,K_R(x_0)}^{q}R^{q(1-\frac{\beta}{2})} - \frac{2^{q-3} q\ \omega_N^{1-\frac{1}{q}}}{q-N}  |\rho|_{q,K_R(x_0)} R^{1-\frac{\beta}{2}}\nonumber\\
 &&\displaystyle - \left(\frac{3}{2}\right)^{q-1}\frac{2^{q-4} q \ \omega_N^{-1+\frac{1}{q}}}{(q-N)^{q}(2q-1)} |\rho|_{q,K_R(x_0)}^{2q-1}R^{(1-\frac{\beta}{2})(2q-1)} - \frac{2^{q-4}}{(q-N)^{q}}  |\rho|_{q,K_R(x_0)}^{q}R^{q(1-\frac{\beta}{2})}. 
\end{eqnarray}
For the first integral in the right-hand side of \eqref{eq:sviluppogradestim}, we claim that there exists a positive constant $c=c(N,m)$ depending only on $N$, $m$ such that
  \begin{equation}\label{eq3bismainteo}
 R^{-N} \int_{K_{R}(x_0)} v^{\gamma+1} \ dx \geq \frac{\omega_N^{\gamma+2}}{\left(\omega_N+ c R^{-N}|\rho|_m^{\frac{Nm}{N-m}}\right)^{\gamma+1}}.
 \end{equation}
Since $\rho \in L^m(\RN)\subset \X^*$, \cite[Proposition 2.6]{BDP} implies that $u$ coincides with the unique minimizer of the energy $I_\rho$ in $\X$. Hence, by the minimality of $I_\rho(u)$, and since for $0\in \X$ one has $I_\rho(0)=0$, we deduce that $I_\rho(u)\leq 0$. Using this, the elementary inequality $\frac{1}{2}t \leq 1-\sqrt{1-t}$, for all $t\in[0,1]$, H\"older's inequality and Lemma \ref{lem:tech2}, we obtain 
\begin{equation}\label{eq0:lemtech2} 
\displaystyle \frac{1}{2}|\n u|^2_2 \le\displaystyle \irn \Big( 1- \sqrt{1-|\n u|^2}\Big) \ dx \le\displaystyle \langle \rho ,u \rangle \leq \displaystyle |\rho|_{m} | u|_{m^\prime} \leq\displaystyle c |\rho|_{m} |\n u|_{2}^{2 \frac{(N+1)m-N}{mN}} 
\end{equation}
where $m^\prime=\frac{m}{m-1}$ is the conjugate exponent of $m$ and $c=c(N,m)$ is a positive constant depending only on $N$, $m$. Then, the inequality \eqref{eq0:lemtech2} implies that 
\begin{equation}\label{eq3quatermainteo}
|\n u|_2 \leq \left(2 c |\rho|_{m}\right)^{\frac{mN}{2(N-m)}}.
\end{equation} 
Now, taking $u \in \X$ in the weak formulation of \eqref{eq:BI}, we get
\begin{equation}\label{eq:applweaksol}
\int_{\RN}  \frac{  |\n  u|^2}{ \sqrt{1- |\n u |^2}}\, dx = \int_{\RN} \rho u \, dx.
\end{equation}
Since $\rho \in L^m(\RN)$ and thanks to H\"older's inequality, Lemma \ref{lem:tech2} and \eqref{eq3quatermainteo}, we deduce that
 \begin{equation}\label{eq5mainteo}
\left|\irn \rho u \ dx\right| \leq  |\rho|_{m} |u|_{m^\prime} \leq 2^{\frac{(N+1)m-N}{N-m}} \left(c|\rho|_{m}\right)^{m^*},
\end{equation}
where $m^*=\frac{Nm}{N-m}$, and thus
 \begin{equation}\label{eq:applweaksolestimate}
\int_{\RN}  \frac{  |\n  u|^2}{ \sqrt{1- |\n u |^2}}\, dx \leq c_1|\rho|_{m}^{m^*},
\end{equation} 
with $c_1:=2^{\frac{(N+1)m-N}{N-m}} c^{m^*}$ (which is a positive constant depending only on $N$, $m$). 
Starting from the identity
$$\int_{B_R(x_0)}  \frac{  1}{ \sqrt{1- |\n u|^2}}\, dx = \int_{B_R(x_0)}  \sqrt{1- |\n u|^2}\,  dx+  \int_{B_R(x_0)}  \frac{  |\n  u|^2}{ \sqrt{1- |\n u|^2}}\, dx, $$
and using the fact that $\sqrt{1- |\n u|^2}\leq 1$, we deduce from \eqref{eq:applweaksolestimate} that
\begin{equation}\label{eq6mainteo}
\int_{B_R(x_0)}  \frac{  1}{ \sqrt{1- |\n u|^2}}\, dx\leq \omega_N R^N + c_1 |\rho|_{m}^{m^*}.
\end{equation}
Finally, since $u$ is strictly spacelike, we can write, recalling that $v=\sqrt{1-|\n u|^2}>0$  in $\RN$,
\begin{eqnarray}\label{eq:auxprop3}
\displaystyle \omega_N R^N  
&\leq &\displaystyle\left( \int_{B_R(x_0)} v^{\gamma+1} \ dx\right)^{\frac{1}{\gamma+2}} \left(\int_{B_R(x_0)} v^{-1} \ dx\right)^{\frac{\gamma+1}{\gamma+2}}\nonumber \\
& \leq&\displaystyle  \left( \int_{B_R(x_0)} v^{\gamma+1} \ dx\right)^{\frac{1}{\gamma+2}}  \left(\omega_N R^N+ c_1|\rho|_{m}^{m^*}\right)^{\frac{\gamma+1}{\gamma+2}},
\end{eqnarray}
where we used H\"older's inequality and \eqref{eq6mainteo}. 
Since, $B_R(x_0)\subset K_R(x_0)$ by definition, the claim \eqref{eq3bismainteo} now follows from \eqref{eq:auxprop3}.

At last, plugging  \eqref{eq3bismainteo} inside \eqref{eq:sviluppogradestim} leads to the estimate
\begin{eqnarray}\label{eq:sviluppogradestim2}
\displaystyle  v^\gamma(x_0)&\geq & \displaystyle \frac{\omega_N^{\gamma+1}}{\left(\omega_N+ c R^{-N}|\rho|_m^{\frac{Nm}{N-m}}\right)^{\gamma+1}} -  \left(\frac{3}{2}\right)\frac{2^{q-4} q\ \omega_N^{-\frac{2}{q}}}{q-N}|\rho|_{q,K_R(x_0)}^2R^{2(1-\frac{\beta}{2})}  \nonumber\\
&&\displaystyle   - \left(\frac{3}{2}\right)^{q-1}\frac{2^{q-5} q \ \omega_N^{-2+\frac{2}{q}}}{(q-N)^{q-1}(q-1)} |\rho|_{q,K_R(x_0)}^{2(q-1)}R^{(2-\beta)(q-1)} \nonumber\\
 &&\displaystyle -  \frac{2^{q-4} \omega_N^{-1}}{(q-N)^{q-1}} \left(\frac{3}{2}+\frac{1}{q-N}\right)  |\rho|_{q,K_R(x_0)}^{q}R^{q(1-\frac{\beta}{2})} - \frac{2^{q-3} q\ \omega_N^{-\frac{1}{q}}}{q-N}  |\rho|_{q,K_R(x_0)} R^{1-\frac{\beta}{2}}\nonumber\\
 &&\displaystyle - \left(\frac{3}{2}\right)^{q-1}\frac{2^{q-4} q \ \omega_N^{-2+\frac{1}{q}}}{(q-N)^{q}(2q-1)} |\rho|_{q,K_R(x_0)}^{2q-1}R^{(1-\frac{\beta}{2})(2q-1)}.
\end{eqnarray}
As $1-\frac{\beta}{2}=\frac{q-N}{q}$, the estimate \eqref{eq:sviluppogradestim2} yields \eqref{eq:thesisprop:mainestimateCOR0} with $P$ defined by \eqref{eq:defPk}.
\end{proof}

If we assume furthermore that $\rho \in L^q(\RN)\setminus\{0\}$ in the statement of Proposition \ref{prop:mainestimate0}, then it is possible to deduce a global bound independent of the radius and the base point. More precisely we can absorb them in a free parameter $k>0$. This is the content of the next result.

\begin{proposition}\label{prop:mainestimate1}
 Let $N\geq 3$, $\rho \in (L^q(\RN)\cap L^m(\RN) \cap C^1(\RN))\setminus\{0\}$, with $q>N$, $m\in(1,2_*]$. Assume $u \in \X\cap C^3(\RN)$ is a strictly spacelike weak solution of \eqref{eq:BI}. There exist two positive constants $\gamma \in (0,\frac{1}{N})$ depending only on $N$, $c$  depending only on $N$ and $m$, such that for any $k>0$ one has
  \begin{equation}\label{eq:thesisprop:mainestimateCOR}
\begin{array}{lll}
\displaystyle  \inf_{\RN}(1-|\nabla u|^2)^{\frac{\gamma}{2}}&\geq & \displaystyle \frac{\omega_N^{\gamma+1}}{\left(\omega_N+ c |\rho|_m^{\frac{Nm}{N-m}}|\rho|_q^{\frac{Nq}{q-N}} k^{-\frac{Nq}{q-N}}\right)^{\gamma+1}} - P(k),
\end{array}
\end{equation}
where $P(k)$ is given by \eqref{eq:defPk}
\end{proposition}
\begin{proof}
The proof is identical to that of Proposition \ref{prop:mainestimate0} up to \eqref{eq:sviluppogradestim2}, with the only caution that here $|\rho|_{q,K_R(x_0)}$ is replaced by $|\rho|_q$. Set $k:=R^{1-\frac{\beta}{2}}|\rho|_q>0$ and observe that $R^{-N}=k^{-\frac{Nq}{q-N}}|\rho|_q^{\frac{Nq}{q-N}}$. Since $R>0$ is arbitrary, we deduce from \eqref{eq:sviluppogradestim2} that for any $k>0$,
\begin{equation}\label{eq:thesisprop:mainestimate1}
\begin{array}{lll}
\displaystyle  v^\gamma(x_0)&\geq & \displaystyle \frac{\omega_N^{\gamma+1}}{\left(\omega_N+ c |\rho|_m^{\frac{Nm}{N-m}}|\rho|_q^{\frac{Nq}{q-N}} k^{-\frac{Nq}{q-N}}\right)^{\gamma+1}} - P(k),
\end{array}
\end{equation}
where $P(k)$ is given by \eqref{eq:defPk}. Finally, since the constants $\gamma$, $c$, as well as $P(k)$ are independent of $x_0 \in \RN$ (in particular the right-hand side of \eqref{eq:thesisprop:mainestimate1} does not depend on $x_0$) then we readily obtain \eqref{eq:thesisprop:mainestimateCOR}.
\end{proof}

\begin{remark}\label{rem1}
We point out that \eqref{eq:thesisprop:mainestimateCOR} is invariant under the transformation $u\mapsto \tilde u_t$, $\rho\mapsto \bar\rho_t$, where $\tilde u_t(x):=t u \left(\frac{x}{t}\right)$, $\bar\rho_t(x):=\frac{1}{t}\rho\left(\frac{x}{t}\right)$ and $t$ is a given positive number. Such a transformation is naturally associated to Problem \eqref{eq:BI} in the sense  that $u\in\X$ if and only if $\tilde u_t \in \X$  and $u$ is a solution to \eqref{eq:BI} if and only if $\tilde u_t$ is a solution of the same problem with datum $\bar\rho_t$.
\end{remark}

\begin{remark}\label{rem2}
A crucial property used in the proofs of Proposition \ref{prop:mainestimate0}, Proposition \ref{prop:mainestimate1} (see \eqref{eq0:lemtech2}, \eqref{eq5mainteo}) and in that of Lemma \ref{lem:tech2} is the continuous embedding of the space $\X \hookrightarrow L^{p^*}(\RN)$, for all $p\in[2,N)$. This reflects on the fact that for $u\in \X$, $\rho \in L^m(\RN)$, $m\geq 1$, the duality pairing
\begin{equation}\label{eq:dualitypair}
\langle \rho, u \rangle =\int_{\RN} \rho u \ dx
\end{equation}
makes sense only for $m\in[1,2_*]$.
\end{remark}

We now analyze the case $m=1$. The counterpart of Proposition \ref{prop:mainestimate0} is the following 

\begin{proposition}\label{prop:mainestimate20}
 Let $N\geq 3$, $q>N$ and $\rho \in L^1(\RN) \cap C^1(\RN)$. Assume $u \in \X\cap C^3(\RN)$ is a strictly spacelike weak solution of \eqref{eq:BI}. For any given $s>N$ there exist three positive constants  $\gamma$, $c_1$, $c_2$ with $\gamma \in (0,\frac{1}{N})$ depending only on $N$, $c_1$, $c_2$  depending only on $N$ and $s$, such that for any $x_0\in\RN$, $R>0$ one has
 \begin{equation*}
\displaystyle   (1-|\nabla u|^2)^{\frac{\gamma}{2}} (x_0) \geq  \displaystyle 
\begin{cases} \frac{\omega_N^{\gamma+1}}{\left(\omega_N+ c_1 R^{-N} |\rho|_1^{\frac{s}{s-1}}\right)^{\gamma+1}} -  P\left(|\rho|_{q,K_R(x_0)}R^{\frac{q-N}{q}}\right) & \hbox{if $\|u\|_{\X} \leq 1$},\\
 \frac{\omega_N^{\gamma+1}}{\left(\omega_N+ c_2 R^{-N}  |\rho|_1^{\frac{Ns}{Ns-(N+s)}} \right)^{\gamma+1}} -  P\left(|\rho|_{q,K_R(x_0)}R^{\frac{q-N}{q}}\right) & \hbox{if $\|u\|_{\X} > 1$},
\end{cases}
\end{equation*}
where $\omega_N$ is the volume of the unit ball in $\R^N$ and $P=P(k)$ is given by \eqref{eq:defPk}.
\end{proposition}
\begin{proof}
In this proof, any $c_i$, $i\in\N$ denotes a positive constant depending only on $N$ and $s$.
The argument is similar to the one used in the proof of Proposition \ref{prop:mainestimate0}, but some adjustments are needed for Claim \ref{eq3bismainteo}. To this end, fixing $s>N$, arguing as in \eqref{eq0:lemtech2} and exploiting Lemma \ref{lem:tech3}, it follows that
\begin{equation}\label{eq1:prop:mainestimate2}
\displaystyle \frac{1}{2}|\n u|^2_2  \leq \displaystyle |\rho|_{1} | u|_{\infty} \leq\displaystyle c_1 |\rho|_{1} \left( c_2|\n u|_{2}^{2 \frac{N+s}{Ns}} + |\n u|_{2}^{\frac{2}{s}} \right).
\end{equation}
Assume that $\|u\|_{\X}\leq 1$ (i.e. $|\nabla u|_2\leq 1$). Then from \eqref{eq1:prop:mainestimate2} and since $\frac{N+s}{Ns}>\frac{1}{s}$ we get 
\begin{equation}\label{eq2:prop:mainestimate2}
|\n u|_2  \leq c_3 |\rho|_{1}^{\frac{s}{2(s-1)}}.
\end{equation}
Combining \eqref{eq2:lemtech2}, \eqref{eq2:prop:mainestimate2} we have
\begin{equation}\label{eq3:prop:mainestimate2}
|u|_\infty \leq c_4 |\rho|_1^{\frac{1}{s-1}}.
\end{equation}
From \eqref{eq3:prop:mainestimate2} we deduce that
\begin{equation}\label{eq3bis:prop:mainestimate2}
\left|\irn \rho u \ dx\right| \leq  |\rho|_{1} |u|_{\infty} \leq c_4 |\rho|_1^{\frac{s}{s-1}},
\end{equation}
 and thus, arguing as in \eqref{eq:applweaksol}--\eqref{eq:auxprop3}, we infer that for any $x_0\in\RN$, $R>0$, it holds that
$$R^{-N} \int_{K_{R}(x_0)} v^{\gamma+1} \ dx \geq \frac{\omega_N^{\gamma+2}}{\left(\omega_N+ c_5 R^{-N}|\rho|_1^{\frac{s}{s-1}}\right)^{\gamma+1}}.$$
This gives the counterpart of Claim \ref{eq3bismainteo} and this completes the proof of \eqref{eq:thesisprop:mainestimate2} when $\|u\|_{\X}\leq 1$.

On the other hand, if $\|u\|_{\X}>1$, fixing $s>N$, using \eqref{eq1:prop:mainestimate2} and taking into account that $\frac{N+s}{Ns}>\frac{1}{s}$ and $(N-1)s>N$, we infer that
\begin{equation}\label{eq4:prop:mainestimate2}
|\n u|_2  \leq c_6 |\rho|_{1}^{\frac{Ns}{2(Ns-N-s)}}.
\end{equation}
Hence, as in the previous case, from \eqref{eq2:lemtech2} and \eqref{eq4:prop:mainestimate2} one concludes that
\begin{equation}\label{eq5:prop:mainestimate2}
|u|_\infty \leq c_7 |\rho|_1^{\frac{N+s}{Ns-N-s}},
\end{equation}
and arguing as in \eqref{eq3bis:prop:mainestimate2} and \eqref{eq:applweaksol}--\eqref{eq:auxprop3}, it follows that
$$R^{-N} \int_{K_{R}(x_0)} v^{\gamma+1} \ dx \geq \frac{\omega_N^{\gamma+2}}{\left(\omega_N+ c_8 R^{-N}|\rho|_1^{\frac{Ns}{Ns-N-s}}\right)^{\gamma+1}},$$
for any $x_0\in\RN$, $R>0$.
\end{proof}
Finally, the counterpart of Proposition \ref{prop:mainestimate1} for $m=1$ goes as follows.

\begin{proposition}\label{prop:mainestimate2}
 Let $N\geq 3$, $\rho \in L^q(\RN)\cap L^1(\RN) \cap C^1(\RN)$, with $q>N$. Assume $u \in \X\cap C^3(\RN)$ is a strictly spacelike weak solution of \eqref{eq:BI}. For any given $s>N$ there exist three positive constants  $\gamma$, $c_1$, $c_2$ with $\gamma \in (0,\frac{1}{N})$ depending only on $N$, $c_1$, $c_2$  depending only on $N$ and $s$, such that for any $k>0$ one has
 \begin{equation}\label{eq:thesisprop:mainestimate2}
\displaystyle    \inf_{\RN}(1-|\nabla u|^2)^{\frac{\gamma}{2}} \geq  \displaystyle 
\begin{cases} \frac{\omega_N^{\gamma+1}}{\left(\omega_N+ c_1 |\rho|_1^{\frac{s}{s-1}}|\rho|_q^{\frac{Nq}{q-N}} k^{-\frac{Nq}{q-N}}\right)^{\gamma+1}} - P(k) & \hbox{if $\|u\|_{\X} \leq 1$},\\
 \frac{\omega_N^{\gamma+1}}{\left(\omega_N+ c_2 |\rho|_1^{\frac{Ns}{Ns-(N+s)}}|\rho|_q^{\frac{Nq}{q-N}} k^{-\frac{Nq}{q-N}}\right)^{\gamma+1}} - P(k) & \hbox{if $\|u\|_{\X} > 1$},
\end{cases}
\end{equation}
where $\omega_N$ is the volume of the unit ball in $\R^N$ and $P(k)$ is given by \eqref{eq:defPk}.
\end{proposition}
\begin{proof}
The proof is identical to that of Proposition \ref{prop:mainestimate1} with the necessary adjustments pointed out in the proof of Proposition \ref{prop:mainestimate20}.
\end{proof}

\begin{remark}
We observe that, differently from \eqref{eq:thesisprop:mainestimateCOR}, the condition in \eqref{eq:thesisprop:mainestimate2} is not invariant under the natural transformation $u\mapsto \tilde u_t$, $\rho\mapsto \bar\rho_t$ described in Remark \ref{rem1}. More precisely, for any $t>0$ one has that 
$$\|\tilde u_t\|_{\X}^2=t^N \|\tilde u\|_{\X}^2, \ |\bar \rho_t|_1^{\frac{s}{s-1}}|\bar \rho_t|_q^{\frac{Nq}{q-N}}=t^{\frac{N-s}{s-1}} |\rho|_1^{\frac{s}{s-1}}|\rho|_q^{\frac{Nq}{q-N}}$$ and   $$|\bar \rho_t|_1^{\frac{Ns}{Ns-(N+s)}}|\bar \rho_t|_q^{\frac{Nq}{q-N}}=t^{\frac{N^2}{Ns-(N+s)}} |\rho|_1^{\frac{Ns}{Ns-(N+s)}}|\rho|_q^{\frac{Nq}{q-N}}.$$
\end{remark}

\section{$W^{2,q}_{loc}$ regularity of the minimizer of the Born-Infeld energy}
In this final section, we prove Theorem \ref{teo:w2qregmin}.
\begin{proof}[Proof of Theorem \ref{teo:w2qregmin}]
Let $N\geq 3$, $\rho \in L^q(\RN)\cap L^m(\RN)$, with $q>N$, $m\in[1,2_*]$ and let $u_\rho \in \X$ be the unique minimizer of $I_\rho$. Let $(\rho_{n})_{n} \subset C^\infty_c(\RN)$ be a standard sequence of mollifications of $\rho$ and consider the sequence $(u_n)_n \subset \X$, where $u_n$ is the unique minimizer of $I_{\rho_n}$ for any $n\in \mathbb{N}$. From \cite[Theorem 1.5, Remark 3.4]{BDP} we know that for any $n\in \mathbb{N}$ the function $u_n$ belongs to $C^\infty(\RN)$, it is strictly spacelike in $\RN$, and it is a weak (and also classical) solution to \eqref{eq:BI} with datum $\rho_{n}$.
Moreover, since $\rho_n\to \rho$ in $L^m(\RN)$, as $n\to +\infty$, then from \cite[Theorem 5.3, Corollary 5.4 and Remark 5.5]{BDP}, up to a subsequence, as $n\to +\infty$ we have that $(u_n)_n$ converges to $u_\rho$ weakly in $\X$ and uniformly in $\RN$. 

We fix once for all two positive numbers $s>N$ and $\overline R>0$. We divide the proof in successive steps and even if this is tedious, we keep track of the constants and on the parameters on which they depend. 
 
 \textbf{Step 1:} there exists a positive constant $C$ depending only on $N$, $m$, $s$ and $|\rho|_m$ such that
\begin{equation}\label{eq:unifboundun}
|u_n|_\infty \leq C \ \ \forall n\in \N.\\[12pt]
\end{equation}
 
We first notice that Step 1 holds true, as $u_n\to u_\rho$ uniformly in $\RN$. Nevertheless, for the sake of completeness, we give a direct proof of \eqref{eq:unifboundun}.  Indeed, if $m\in(1,2_*]$ then thanks to \eqref{eq3quatermainteo} and since $|\rho_n|_m\leq |\rho|_m$, which is a standard property of the mollified sequence, we get
\begin{equation}\label{eq0bis:proofmainteo}
|\n u_n|_2\leq c_1 |\rho|_{m}^{\frac{mN}{2(N-m)}},
\end{equation}
for some positive constant $c_1$ depending only on $N$ and $m$. Then from Lemma \ref{lem:tech3} and \eqref{eq0bis:proofmainteo} we readily infer that
 \begin{equation}\label{eq1:proofmainteo}
|u_n|_\infty \leq c_2 |\rho|_{m}^{\frac{m(N+s)}{(N-m)s}} + c_3  |\rho|_{m}^{\frac{mN}{(N-m)s}},
\end{equation}
where $c_2$, $c_3$ are positive constants depending only on $N$, $s$, $m$ and we are done. Similarly, when $m=1$, then, taking into account \eqref{eq2:prop:mainestimate2}--\eqref{eq3:prop:mainestimate2} (if $|\n u_n|_2\leq 1$ ) or \eqref{eq4:prop:mainestimate2}--\eqref{eq5:prop:mainestimate2} (if $|\n u_n|_2> 1$), we deduce that
 \begin{equation}\label{eq2:proofmainteo}
|u_n|_\infty \leq \max\left\{c_4 |\rho|_{1}^{\frac{1}{s-1}}; c_5  |\rho|_{1}^{\frac{N+s}{Ns-N-s}}\right\},
\end{equation}
where $c_4$, $c_5$ are positive constants depending only on $N$ and $s$.
At the end, from \eqref{eq1:proofmainteo} (or \eqref{eq2:proofmainteo} if $m=1$) we obtain \eqref{eq:unifboundun}, with a positive constant $C$ depending only on $s$, $|\rho|_m$, $N$ and $m$.\\

\textbf{Step 2:} there exists a constant $\delta\in[0,1)$ depending only on $\overline R$, $N$, $m$, $|\rho|_m$ (and $s$ if $m=1$) such that for any $n\in \N$ there exists a bounded open neighborhood ${W}_n$ of $\supp(\rho_n)$ ($\supp(\rho_n)$ is the support of $\rho_n$) such that
\beq\label{eq:finalclaim}
|\nabla u_n|_{\infty,{W}_n^\complement} \leq \delta.\\[12pt]
\eeq

Let $\overline R$ be the positive number chosen at the beginning of the proof and denote by $K_{\overline R}^n(x_0)$ the projection of the Lorentz ball associated to $u_n$, centred at $x_0\in \RN$ and of radius $\overline R$.  Thanks to Lemma \ref{lem:boundedLorentzballX}-(i) and \eqref{eq:unifboundun}, for any $x_0\in \RN$, for any $n\in \N$ we have $K_{\overline R}^n(x_0)\subset B_{R^\prime}(x_0)$, with $R^\prime=\sqrt{\overline{R}^2+4C}$. Hence, denoting by $\dist(x,\supp(\rho_n))$  the Euclidean distance between $x \in \RN$ and $\supp(\rho_n)$, and setting $${W}_n:=\{x \in \RN; \ \dist(x,\supp(\rho_n))<R^\prime\},$$ 
 we infer that $K_{\overline{R}}^n(x_0) \cap \supp(\rho_n) = \emptyset$ for any $x_0\in {W}_n^\complement$, for all $n\in \N$. In particular, $|\rho_n|_{q,K_{\overline R}^n(x_0)}=0$ and, if $m\in(1,2_*]$, then, by Proposition \ref{prop:mainestimate0} and taking into account that  $|\rho_n|_m\leq |\rho|_m$, we deduce that
\beq\label{eq:preclaim}
 |\nabla u_n(x_0)|^2 \leq 1 -\frac{\omega_N^{\frac{2(\gamma+1)}{\gamma}}}{\left(\omega_N+ c_6 \overline R^{-N}|\rho|_m^{\frac{Nm}{N-m}}\right)^{\frac{2(\gamma+1)}{\gamma}}} \ \ \ \forall x_0\in {W}_n^\complement, \ \forall n\in \N,
\eeq
where $\gamma \in(0,\frac{1}{N})$ depends on $N$ only, and $c_6$ depends on $N$, $m$ only. Similarly, when $m=1$,  using Proposition \ref{prop:mainestimate20} we deduce that for any $x_0\in {W}_n^\complement$, for any $n\in \N$ it holds
\beq\label{eq:preclaim2}
 |\nabla u_n(x_0)|^2 \leq 1 - \min\left\{ \frac{\omega_N^{\frac{2(\gamma+1)}{\gamma}}}{\left(\omega_N+ c_7 \overline R^{-N} |\rho|_1^{\frac{s}{s-1}}\right)^{\frac{2(\gamma+1)}{\gamma}}},  \frac{\omega_N^{\frac{2(\gamma+1)}{\gamma}}}{\left(\omega_N+ c_8 \overline R^{-N}  |\rho|_1^{\frac{Ns}{Ns-(N+s)}} \right)^{\frac{2(\gamma+1)}{\gamma}}} \right\},
 \eeq
 for some positive constants $c_7$, $c_8$ depending only on $N$ and $s$, and $\gamma$ is as in the previous case.
 
Finally, Step 2 follows immediately from \eqref{eq:preclaim} (or \eqref{eq:preclaim2} if $m=1$), with $\delta \in [0,1)$ depending only on $N$, $m$, $|\rho|_m$ and $\overline R$ (or on $N$, $s$, $|\rho|_1$ and $\overline R$, if $m=1$).\\

\textbf{Step 3:} there exists a constant $\theta \in [0,1)$ independent of $n$ such that

\beq\label{eq:Haa3}
 |\nabla u_n|_\infty \leq \theta \ \ \forall n \in \N.\\[12pt]
\eeq

Let us set
 $$\nu_n(x):=\frac{1}{\sqrt{1-|\nabla u_n(x)|^2}}, \ \ \ x\in \RN,$$
and observe that proving Step 3 is equivalent to finding a constant $\bar\nu\geq 1$ independent of $n$ such that $$ \sup_{\RN} \nu_n \leq \bar\nu \ \ \ \forall n\in \N.$$ 
To this end, let $\delta\in[0,1)$ be the number given by Step 2 and fix a number $\nu_0>\frac{1}{\sqrt{1-\delta^2}}\geq 1$. Then by definition and Step 2 it follows that $(\nu_n-\nu_0)_+$ has compact support for all $n\in \N$, where $(\cdot)_+$ stands for the positive part of a function. This means that the assumptions of Proposition \ref{prop:Lqestnu} are satisfied by $u_n$, for all $n\in\N$, with a uniform value of $\nu_0$ independent of $n\in \N$. Hence, by \eqref{eq:tesiTh36}, \eqref{eq:unifboundun}, and recalling that $|\rho_n|_q \leq |\rho|_q$, we deduce that
\beq\label{eq:Haa0}
|(\nu_n-\nu_0)_+|_q\leq c_9 |\rho|_q \ \ \ \ \ \forall n\in \N,
\eeq
for some positive constant $c_9$ depending only on $\nu_0$, $s$, $N$, $q$, $m$, $|\rho|_m$. In particular, since $$\nu_n\leq (\nu_n-\nu_0)_+ + \nu_0 \ \ \hbox{in $\RN$,}$$ 
then, for any $x_0 \in \RN$, for any $R>0$, using the monotonicity and triangle inequality of the norm $|\nu_n|_{q, B_R(x_0)}$, and exploiting \eqref{eq:Haa0}, we infer that
$$
|\nu_n|_{q, B_R(x_0)}\leq  c_9|\rho|_q + \nu_0 \omega_N^{\frac{1}{q}} R^{\frac{N}{q}} \ \ \ \ \ \forall n\in \N.
$$
Then
\beq\label{eq:Haa1}
 \left(\Mint_{B_R(x_0)} \nu_n^q \, dx\right)^{\frac{1}{q}} \leq  c_9 \omega_N^{-\frac{1}{q}} R^{-\frac{N}{q}}|\rho|_q + \nu_0 \ \ \ \ \ \forall n\in \N.
\eeq
Finally, thanks to Theorem \ref{teo:maingradestimate2}, we get the estimate
\beq\label{eq:Haa2}
\sup_{B_{R/2}(x_0)} \nu_n \leq c_{10} \left[ \left(\Mint_{B_R(x_0)} \nu_n^q \, dx\right)^{\frac{N}{q(q-N)}} + R^{\frac{N}{q-N}} \left(\Mint_{B_R(x_0)} |\rho|^q \, dx\right)^{\frac{N}{q(q-N)}}\right] \left(\Mint_{B_R(x_0)} \nu_n^q \, dx\right)^{\frac{1}{q}},
\eeq
where $c_{10}$ is a positive constant depending only on $N$ and $q$ and $x_0\in \RN$ is arbitrary.
Hence, combining \eqref{eq:Haa1}, \eqref{eq:Haa2}, we infer that
\beq\label{eq:estimateBnun}
\sup_{B_{R/2}(x_0)} \nu_n \leq \overline\nu \ \ \forall n\in \N, 
\eeq
for some constant $\overline\nu\geq 1$ depending only on $\nu_0$, $s$, $N$, $q$, $m$, $|\rho|_m$, $|\rho|_q$ and $R$. Therefore, as $\bar\nu$ is independent of $x_0$ and $x_0\in \RN$ is arbitrary, then \eqref{eq:estimateBnun} yields
\begin{equation}\label{unifboundnun}
\sup_{\RN} \nu_n \leq \overline\nu \ \ \forall n\in \N. 
\end{equation}
As pointed out before, this is equivalent to \eqref{eq:Haa3} and he proof of Step 3 is complete.\\

\textbf{Step 4:} there exists a number $\alpha \in (0,1)$ independent of $n$ such that for any bounded domain $\Omega\subset \RN$ there exists a positive constant $C_1$ independent of $n$, such that 
\beq\label{eq:Step3}
\|u_n\|_{C^{1,\alpha}(\overline\Omega)}\leq C_1  \ \ \forall n\in \N.\\[6pt]
\eeq

One could prove Step 4 by arguing as in \cite[Proof of Theorem 1.6, Step 5]{BIA}, where essentially the Universal potential estimates (see \cite[Theorem 1.4]{KUM})  is applied to \eqref{eq:BI2}, taking into account Step 3. A second approach, followed in \cite[Corollary 3.2]{Haa}, consists in   differentiating \eqref{eq:BI2} with respect to $x_i$, for $i=1,\ldots,N$, yielding the equation \eqref{eq:BIdiff}. Then Step 3 and standard regularity results give the uniform H\"older continuity of the derivatives $(u_n)_i:=\frac{\partial u_n}{\partial x_i}$. This simpler method is somehow more natural and that is the strategy we will adopt here.\\

Let us fix $i\in\{1,\ldots,N\}$. Differentiating \eqref{eq:BI2} with respect to $x_i$ we get that, for any $n\in\N$, the function $(u_n)_i$
weakly, and also classically, satisfies
\begin{equation}\label{eq:BIdiffn}
-\operatorname{div}\left(D^2F (\nabla u_n) \nabla (u_n)_i\right)=(\rho_n)_i \ \ \hbox{in $\RN$}.
\end{equation}
Recalling \eqref{eq0:unifboundD2F} and since $\nu_n(x)=\frac{1}{\sqrt{1-|\nabla u_n(x)|^2}}\geq 1$ for any $x\in\RN$, $n\in\N$, we see that the equation \eqref{eq:BIdiffn} is uniformly elliptic in $\RN$, with ellipticity constant $\lambda=1$. In addition, thanks to \eqref{eq:unifboundD2F} and Step 3 (see \eqref{unifboundnun}), there exists a positive constant $\overline\Lambda$ such that for any $n\in \N$ it holds
$$\sup_{x\in\RN}|D^2F(\nabla (u_n)_i(x))|\leq \overline\Lambda.$$
Now, for any bounded domain $\Omega$ of $\RN$, and since $q>N$, we can apply \cite[Theorem 8.24]{GT} to the solution $(u_n)_i$ of \eqref{eq:BIdiffn} (restricted to $\Omega$) with $\lambda:=1$, $\Lambda:=\overline\Lambda$, $f^j:=\rho_n\delta_{ij} \in L^q(\Omega)$, for $j=1,\ldots,N$, $g:=0$ and $\nu:=0$ (we refer to \cite[(8.1)--(8.6)]{GT} for the notations). We therefore infer that
$$\|(u_n)_i\|_{C^{0,\alpha}(\overline\Omega^\prime)}\leq C \left(|(u_n)_i|_{2,\Omega}+|\rho_n|_{q,\Omega}\right),$$
where $\Omega^\prime\subset\subset \Omega$ , $C=C(N,q,\Lambda/\lambda, d^\prime)>0$, $d^\prime:=\mathrm{dist}(\Omega^\prime, \partial\Omega)$, and $\alpha=\alpha(N,\Lambda/\lambda)\in(0,1)$. In particular, $C$ and $\alpha$ can be chosen independently of $n$ in this estimate. We point out that the number $\alpha$ is also independent of $\Omega$ and $\Omega^\prime$ (see \cite[Remark at page 202]{GT}). Finally, as $|\rho_n|_{q,\Omega}\leq |\rho|_{q,\Omega}$, $|(u_n)_i|_\infty \leq 1$, we readily deduce that 
$$\|(u_n)_i\|_{C^{0,\alpha}(\overline\Omega^\prime)}\leq C_1,$$\\\
where $C_1=C_1(N,q,\Lambda/\lambda, d^\prime,|\rho|_{q,\Omega})>0$ is independent of $n$ and thus, summing for all $i=1,\ldots,N$ and taking into account Step 1 we obtain \eqref{eq:Step3}. The proof of Step 4 is complete.\\

\textbf{Step 5:} for any bounded domain $\Omega^\prime \subset \RN$ there exists a positive constant $C_2$ independent of $n$ such that
\beq\label{eq:tesiStep4}
\|u_n\|_{W^{2,q}(\Omega^\prime)} \leq C_2 \ \ \ \forall n\in \N.\\[12pt]
\eeq
Indeed, fix two bounded domains $\Omega^\prime \subset \subset \Omega \subset \RN$. Observe that as $u_n$ is a smooth strictly spacelike classical solution to \eqref{eq:BI}, we can write the equation in non-divergence form. This yields
\beq
(1-|\nabla u_n|^2)\Delta u_n+\sum_{i,j=1}^N (u_n)_i(u_n)_j (u_n)_{ij}=-(1-|\nabla u_n|^2)^{\frac{3}{2}}\rho_n.
\eeq
In particular, setting $a_{ij}^n:=\delta_{ij} (1-|\nabla u_n|^2) + (u_n)_i(u_n)_j$, 
and $f_n:=-(1-|\nabla u_n|^2)^{\frac{3}{2}}\rho_n$, we see that $u=u_n$ is a strong solution in $\Omega$ of the  equation
$$\sum_{i,j=1}^N a_{ij}^n(x) u_{ij} = f_n.$$
Clearly, by construction, one has that $|f_n|_q\leq |\rho_n|_q\leq|\rho|_q$ and $|a_{ij}^n(x)|\leq 1$ for all $i,j=1,\ldots,N$, for any $x \in \Omega$, $n \in \N$. Moreover, if $\theta\in[0,1)$ is the constant independent of $n$ given by Step 3, then one can check that
$$ \sum_{i,j=1}^N a_{ij}^n(x)\xi_i\xi_j \geq (1-\theta^2)|\xi|^2 \ \ \ \forall \xi \in \RN\ \forall x \in \Omega.$$
In addition, by construction, thanks to Step 4 and exploiting the elementary properties of H\"olderian functions (namely, if $g_1,g_2 \in C^{0,\alpha}(\overline\Omega)$ then $g_1g_2\in C^{0,\alpha}(\overline\Omega)$) we infer that $\|a_{ij}^n\|_{C^{0,\alpha}(\overline\Omega)} \leq c_{15}$, for some constant $c_{15}>0$ independent of $n$, where $\alpha\in (0,1)$ is the number given by Step 4. In particular, this means that we have a uniform control on the moduli of continuity of the coefficients $a_{ij}^n$ in $\overline\Omega$, with respect to $n\in \N$. Therefore, by standard elliptic regularity theory (see \cite[Theorem 9.11]{GT}) we deduce that
\beq\label{eq:unifboundweqomega}
\|u_n\|_{W^{2,q}(\Omega^\prime)} \leq c_{16}(|u_n|_{q,\Omega}+|f_n|_{q,\Omega}),
\eeq
where $c_{16}$ is a positive constant depending only on $N$, $q$, $\theta$, $\Omega^\prime$, $\Omega$, and the moduli of continuity of the coefficients $a_{ij}^n$ in $\Omega^\prime$. Now, from the proof of \cite[Theorem 9.11]{GT} and since $\|a_{ij}^n\|_{C^{0,\alpha}(\overline\Omega)} \leq c_{15}$, we deduce that $c_{16}$ is independent of $n\in \N$.\\

At the end, from \eqref{eq:unifboundweqomega}, thanks to Step 1 and since $|f_n|_q\leq |\rho|_q$ we readily obtain \eqref{eq:tesiStep4} with a constant independent of $n\in \N$. The proof of Step 5 is complete.\\

\textbf{Conclusion:} given a smooth bounded domain $\Omega^\prime \subset \RN$, from Step 5 it follows that $(u_n)_n$ is a bounded sequence in $W^{2,q}(\Omega^\prime)$ and thus, up to a subsequence, $u_n\rightharpoonup \bar u$ in $W^{2,q}(\Omega^\prime)$, for some $\bar u \in W^{2,q}(\Omega^\prime)$, and as $q>N$, by the Rellich-Kondrachov's theorem, up to a further subsequence, it follows that $u_n\to \bar u$ in $C^1(\overline{\Omega^\prime})$. On the other hand, as $u_n\to u_\rho$ in compact subsets of $\RN$ (actually $u_n\to u_\rho$ uniformly in $\RN$, as pointed out at the beginning of the proof) it follows that $u_\rho=\bar u$ in $\Omega^\prime$, and thus $u_\rho \in W^{2,q}(\Omega^\prime)$. In addition, since $u_n\to u_\rho$ in $C^1(\overline{\Omega^\prime})$ then it also follows that $|\nabla u_\rho|\leq \theta$, where $\theta\in[0,1)$ is the number given by Step 3, which is independent of $n$ and $\Omega^\prime$. In particular, from the arbitrariness of $\Omega^{\prime}$ we infer that $u_\rho$ belongs to $W^{2,q}_{loc}(\RN)$ and that $u_\rho$ is strictly spacelike in $\RN$. To conclude the proof it remains to show that $u_\rho$ is the weak solution to \eqref{eq:BI}.
To this end let us fix $\varphi \in C^\infty_c(\RN)$. Then, up to a subsequence, as $\nabla u_n \to \nabla u_\rho$ uniformly in compact subsets of $\R^N$, and since $\frac{1}{\sqrt{1-|\nabla u_n|^2}} \leq \frac{1}{\sqrt{1-\theta^2}}$, where $\theta\in[0,1)$ is the constant given by Step 3, we conclude that
\beq\label{eq1:proofmainteo3}
\lim_{n \to + \infty} \irn \frac{\n u_n \cdot \n \varphi}{\sqrt{1-|\nabla u_n|^2}}\, dx 
=\irn \frac{\n u_\rho \cdot \n \varphi}{\sqrt{1-|\nabla u_\rho|^2}}\, dx.
\eeq
 On the other hand, as $u_n$ is a weak solution to \eqref{eq:BI} with datum $\rho_n$, then, passing to the limit as $n\to +\infty$ in the definition of weak solution (see \eqref{eq:weakBI}) and since $\rho_n \to \rho$ in $L^m(\R^N)$, we have
\beq\label{eq2:proofmainteo3}
\lim_{n \to + \infty} \irn \frac{\n u_n \cdot \n \varphi}{\sqrt{1-|\nabla u_n|^2}}\, dx
=\lim_{n \to + \infty} \irn \rho_n \varphi\, dx = \irn \rho \varphi\, dx.
\eeq
Hence, equating \eqref{eq1:proofmainteo3} and \eqref{eq2:proofmainteo3} we get that $u_\rho$ satisfies the definition of weak solution for \eqref{eq:BI}, for any test function $\varphi \in C^\infty_c(\RN)$. Finally, as $|\nabla u_\rho|\leq \theta<1$ in $\RN$ then arguing by density we infer that the same conclusion holds true for any test function $\varphi \in \X$. Therefore $u_\rho$ is a weak solution to \eqref{eq:BI} and this completes the proof of Theorem \ref{teo:w2qregmin}.
 \end{proof}

\section{Declarations}
\begin{itemize}
\item The authors have no relevant financial or non-financial interests to disclose.
\item The authors have no competing interests to declare that are relevant to the content of this article.
\item Data sharing not applicable to this article as no datasets were generated or analysed during the current study.
\end{itemize}
\section*{Acknowledgments}
\thanks{We wish to thank Akseli Haarala for the useful discussions on \cite{Haa}. We are grateful to the referees for their valuable comments which helped us to improve our manuscript and in particular the statement of Lemma \ref{lem:boundedLorentzballX} and Step 4 of the proof of Theorem \ref{teo:w2qregmin}.}

\end{document}